\documentclass[oneside,a4paper,11pt,notitlepage]{article}
\usepackage[left=0.8in, right=0.8in, top=1.5in, bottom=1.5in]{geometry}
\usepackage[T1]{fontenc}
\usepackage[utf8]{inputenc} 
\usepackage{tocloft}

\AtBeginDocument{\selectfont}

\usepackage[english]{babel}
\usepackage{float}
\usepackage{amssymb}
\usepackage{lmodern}

\usepackage{amsmath}
\usepackage{amsthm}
\usepackage{bm}
\usepackage{mathtools}
\usepackage{braket}
\usepackage{esint}
\newcommand{\abs}[1]{{\left|#1\right|}}
\newcommand{\norma}[1]{{\left\Vert#1\right\Vert}}

\usepackage{enumerate}
\usepackage{booktabs}
\usepackage{graphicx}
\usepackage{tikz}
\usetikzlibrary{patterns}
\usepackage{multicol}
\usepackage{caption}
\usepackage[skins,theorems]{tcolorbox}
\tcbset{highlight math style={enhanced,
colframe=black,colback=white,arc=0pt,boxrule=1pt}}
\def\XXint#1#2#3{{\setbox0=\hbox{$#1{#2#3}{\int}$}
    \vcenter{\hbox{$#2#3$}}\kern-.5\wd0}}

\usepackage[hyphens]{url}
\usepackage{hyperref}
\usepackage{aliascnt}

\hypersetup{
                colorlinks=true,
                linkcolor={magenta},
                citecolor={cyan},
                breaklinks=true}

\theoremstyle{plain}
\newtheorem{teor}{Theorem}
\numberwithin{teor}{section}
\numberwithin{equation}{section}
\newenvironment{teorema}{\begin{teor}}{\end{teor}}

\newcommand{\IClabel}{%
  \label{IC}%
  \textnormal{\textbf{(IC$_{H^s,X}$)}}%
}

\newcommand{\ICref}{%
  \textnormal{\textbf{(\hyperref[IC]{\textcolor{magenta}{IC$_{H^s,X}$}})}}%
}

\newcommand{\Cslabel}{%
  \label{Cs}%
  \textnormal{\textbf{(C$_{H^s}$)}}%
}

\newcommand{\Csref}{%
  \textnormal{\textbf{(\hyperref[Cs]{\textcolor{magenta}{C$_{H^s}$}})}}%
}

\theoremstyle{definition}
\newaliascnt{defi}{teor}
\newtheorem{defi}[defi]{Definition}
\aliascntresetthe{defi}
\newenvironment{definizione}{\begin{defi}}{\end{defi}}

\theoremstyle{plain}
\newaliascnt{lemma}{teor}
\newtheorem{lemma}[lemma]{Lemma}
\aliascntresetthe{lemma}
 
\theoremstyle{plain}
\newaliascnt{prop}{teor}
\newtheorem{prop}[prop]{Proposition}
\aliascntresetthe{prop}
 
\theoremstyle{plain}
\newaliascnt{conjecture}{teor}

\aliascntresetthe{conjecture}

\theoremstyle{plain}
\newaliascnt{cor}{teor}
\newtheorem{cor}[cor]{Corollary}
\aliascntresetthe{cor}
\newenvironment{corollario}{\begin{cor}}{\end{cor}}
 
\theoremstyle{definition}
\newaliascnt{ex}{teor}

\aliascntresetthe{ex}
 
\theoremstyle{definition}
\newaliascnt{oss}{teor}
\newtheorem{oss}[oss]{Remark}
\aliascntresetthe{oss}
 
\theoremstyle{plain}

\DeclareMathOperator{\diam}{diam}

\DeclareMathOperator{\R}{\mathbb{R}}

\DeclareMathOperator{\Vol}{\text{V}}
\DeclareMathOperator{\vol}{\text{v}}
\makeatletter
\newcommand{\myfootnote}[2]{\begingroup
	\def\@makefnmark{}
	\addtocounter{footnote}{-1}
	\footnote{\textbf{#1} #2}
	\endgroup}
\makeatother
  
     \newcommand{\RR}[0]{\mathbb{R}}

   \newcommand{\derivparzialecompl}[2]{\dfrac{\partial#1}{\partial #2}}

    \newcommand{\norm}[2]{\left\|#1\right\|_{#2}}

    \newcommand{\om}[0]{\Omega}

\usepackage{hyperref}
\hypersetup{linktoc=none, bookmarksnumbered, colorlinks=true, linkcolor=magenta, citecolor=cyan}

\title{On some functionals involving torsional rigidity, principal eigenvalue and perimeter}
\author{Vincenzo Amato, Carlo Nitsch, Cristina Trombetti, and Federico Villone}
\date{}

\newcommand{\Addresses}{{ 
 \bigskip
 \textit{E-mail address}, C.~Nitsch: \texttt{c.nitsch@unina.it}

  \medskip 
 
 \noindent\textit{E-mail address}, C.~Trombetti (corresponding author): \texttt{cristina@unina.it} 
   \medskip

 \noindent\textsc{Dipartimento di Matematica e Applicazioni ``R. Caccioppoli'', Universit\`a degli studi di Napoli Federico II, Via Cintia, Complesso Universitario Monte S. Angelo, 80126 Napoli, Italy.}
 
 \medskip
 
 \noindent\textit{E-mail address}, V. ~Amato: \texttt{v.amato@ssmeridionale.it} 

     \medskip 

 \noindent \textit{E-mail address}, F.~Villone: \texttt{f.villone@ssmeridionale.it} 
  
   \medskip 
 
  \noindent\textsc{Mathematical and Physical Sciences for Advanced Materials and Technologies, Scuola Superiore Meridionale, Largo San Marcellino 10, 80138 Napoli, Italy. }

 \par\nopagebreak 

}} 

\begin{document}

\maketitle

   \begin{abstract}
In this paper we study some relationships between the first Dirichlet eigenvalue $\Lambda(\Omega)$ and the torsional rigidity $T(\Omega)$ of a domain $\Omega$. We consider the problem of optimizing the product $\Lambda(\Omega)T(\Omega)$ among sets with prescribed perimeter, both in the class of open sets with finite perimeter and within the class of convex domains. 

We also present local results for the quantity $\Lambda(\Omega)T(\Omega)^q$, with $q>0$, under either a volume or a perimeter constraint.

\textsc{Keywords:P\'olya functional, Shape optimization, Stability in shape optimization}  \\
\textsc{MSC 2020: 49Q10, 35J25, 49K20, 35J05.}  
\end{abstract}

\begin{center}
\begin{minipage}{10cm}
\small
\tableofcontents
\end{minipage}
\end{center}
\section{Introduction}
We study how the shape of a domain influences two fundamental quantities associated with the Laplacian: the \emph{torsional rigidity}, which measures the resistance of a domain to twisting forces, and the \emph{first Dirichlet eigenvalue}, corresponding to its principal vibration frequency. These quantities capture complementary aspects of a domain's geometry: torsional rigidity relates to elasticity and mechanical stability, while the first eigenvalue is linked to wave propagation and heat diffusion. Understanding their interplay has been a central question in shape optimization since the early 20th century, notably in the seminal work of P\'olya \cite{polya1947elasticite}. 

In our paper, we focus on the problem of optimizing their product among domains with a prescribed perimeter.
For a measurable set $\Omega \subset \mathbb{R}^m$, we define its De Giorgi perimeter as
\[
P(\Omega) := \sup \left\{ \int_\Omega \mathrm{div}\, \varphi \, dx : \varphi \in C_c^\infty(\mathbb{R}^m),\ \|\varphi\|_\infty \le 1 \right\},
\]
and we denote by $V(\Omega)$ its Lebesgue measure. By the isoperimetric inequality, any set with finite perimeter also has finite measure.

For $\om\subset\RR^m$ with finite measure, the torsional rigidity of $\om$ is defined as 
\begin{equation}\label{variationa characteriz torsion}
T(\om) := \max_{\varphi \in H_0^1(\om) \setminus \{0\}} \dfrac{\displaystyle \left( \int_\om \varphi \, dx \right)^2}{\displaystyle\int_\om |\nabla \varphi|^2 \, dx},
\end{equation}
or, equivalently,
\[
T(\om) = \int_\om w\, dx,
\]
where $w\in H_0^1(\om)$ is the unique solution to the Poisson problem
\begin{equation}\label{torsion problem}
\begin{cases}
-\Delta w = 1 & \text{in } \om, \\
w = 0 & \text{on } \partial\om.
\end{cases}
\end{equation}
The torsional rigidity of a set $\om$ satisfies the scaling property
$$T(t\om)= t^{m+2}T(\om) \qquad \forall t>0.$$
By the Saint–Venant inequality, conjectured in \cite{SaintVenant1856} and proved in \cite{PTR48}, the torsional rigidity is maximized by balls among domains of fixed volume. As a consequence, balls also maximize the torsional rigidity  among domains with prescribed perimeter. In a scaling-invariant form, this reads as
\begin{equation}
\label{sanvenper}
P(\Omega)^{-\frac{m+2}{m-1}} T(\Omega) \le P(B)^{-\frac{m+2}{m-1}} T(B),
\end{equation}
where $B$ denotes a ball.

Similarly, the first Dirichlet eigenvalue of the Laplacian on $\Omega$ is defined as
\[
\Lambda(\om) := \min_{\varphi \in H_0^1(\om) \setminus \{0\}} \dfrac{\displaystyle\int_\om |\nabla \varphi|^2 \, dx}{\displaystyle\int_\om \varphi^2 \, dx},
\]
or equivalently, it can be characterized as the smallest $\lambda > 0$ such that the boundary value problem
\[
\begin{cases}
-\Delta u = \lambda u & \text{in } \om, \\
u = 0 & \text{on } \partial\om
\end{cases}
\]
admits a non-trivial solution $u \in H_0^1(\Omega)$.

The first Dirichlet eigenvalue of a set $\om$ satisfies the scaling property
\[
\Lambda(t\Omega) = t^{-2} \Lambda(\Omega), \qquad \forall t>0,
\]
By the Faber–Krahn inequality, conjectured in \cite{Rayleigh1877} and proved in \cite{Faber1923} and \cite{Krahn1925} ,the first Dirichlet eigenvalue is minimized by balls among domains of fixed volume. As a consequence, balls also minimize the first eigenvalue among domains with prescribed perimeter. In a scaling-invariant form, this reads as
\begin{equation}
\label{faberkranper}
P(\Omega)^{\frac{2}{m-1}} \Lambda(\Omega) \ge P(B)^{\frac{2}{m-1}} \Lambda(B).
\end{equation}

As mentioned above, the main goal of this paper is to optimize the product  
\[
T(\Omega)\,\Lambda(\Omega)
\]
under perimeter constraint.  
Equivalently, for each $\Omega \in \mathcal{O}$, where 
\begin{equation}
    \label{openperimeter}
    \mathcal{O} := \{\text{open subsets of } \mathbb{R}^m \text{ with finite perimeter}\},
\end{equation}
we consider the scale-invariant functional  
\begin{equation}
    \label{function_G}
    G(\Omega) := \frac{T(\Omega)\,\Lambda(\Omega)}{P(\Omega)^{\frac{m}{m-1}}}.
\end{equation}

By inequalities \eqref{sanvenper} and \eqref{faberkranper}, the optimization of \(G(\Omega)\) is non-trivial, since we can rewrite it as a product:
\[
G(\Omega) = 
\left[P(\Omega)^{-\frac{m+2}{m-1}}\,T(\Omega)\right]
\left[P(\Omega)^{\frac{2}{m-1}}\,\Lambda(\Omega)\right].
\]
Hence, we can observe a competition between the two functionals, which exhibit opposite behaviours on the ball.

The same competition can be observed in many other functionals, see for instance \cite{AGS,BB21,BM23,BG25,LS}, and in particular in the scaling invariant functional
\begin{equation}
    \label{polyafunc}
    F(\om):=\dfrac{\Lambda(\om)T(\om)}{\Vol(\om)}, \qquad\forall\om\in \mathcal{O},
\end{equation}
introduced by P\'olya in \cite{polya1947elasticite}. In the following, we will refer to the functional $F$ also as the P\'olya functional.  The Blaschke-Santal\'o diagram associated with this functional is explored in \cite{LZ22}.

The functional \(G\) introduced in \eqref{function_G} is closely related to the P\'olya functional \(F\) in \eqref{polyafunc}, since
\begin{equation}
    \label{relationgf}
    G(\Omega) = F(\Omega)\, V(\Omega)\, P(\Omega)^{-\frac{m}{m-1}}.
\end{equation}
This relation provides useful insight into the behaviour of \(G\), as several of its properties can be established analogously to those of \(F\).

In particular, in \cite{VBV2015} and \cite{VFNT2016} the authors proved that
\begin{equation}\label{bound F aperti}
    \inf_{\Omega \in \mathcal{O}} F(\Omega) = 0 \qquad \text{and} \qquad \sup_{\Omega \in \mathcal{O}} F(\Omega) = 1,
\end{equation}

and that neither of these extremal values is attained. The latter result, combined with \eqref{relationgf}, leads us to the following proposition.

 \begin{prop}
\label{openopt}

    Let $G$ be the functional defined in \eqref{function_G} and $\mathcal{O}$ be the set defined in \eqref{openperimeter}. Then 
   \begin{equation*}
           \inf_{ \Omega \in \mathcal{ O} }G(\Omega)= 0,   \qquad 
            \sup_{ \Omega \in \mathcal{ O} }G(\Omega) =\displaystyle m^\frac{m}{m-1}\omega_m^{-\frac{1}{m-1}} ,
    \end{equation*}
where $\omega_m$ denotes the volume of the unit ball in $\mathbb{R}^m$. Furthermore, neither of these extremal values is attained within the class $\mathcal{O}$.
\end{prop}

A minimizing sequence can be obtained, for example, by considering sets that become increasingly thin, for which $G(\Omega_n)$ tends to zero.  
Conversely, a maximizing sequence can be constructed by means of a suitable homogenization of a ball. In this way, the functional $F(\Omega_n)$ approaches $1$, and by carefully adjusting the radii of the holes, the corresponding isoperimetric ratio in \eqref{relationgf} can be controlled.

After studying the optimization problem in the class $\mathcal{O}$, it is natural to investigate how the situation may change under additional geometric constraints. In particular, we focus on the subclass of convex domains within $\mathcal{O}$, denoted by
\begin{equation}\label{open and convex}
    \mathcal{C} := \{\text{open, bounded, and convex subsets of } \RR^m\}.
\end{equation}
Moving to the convex class $\mathcal{C}$, several known results on the P\'olya functional already provide significant information about the behavior of $F$. In particular, as proved in \cite[Theorem 1.2]{VFNT2016} and \cite[Remark 4.1]{BM20}, one has
\begin{equation}\label{bound polya}
  \dfrac{\pi^2}{4}\dfrac{1}{m(m+2)} < F(\om) < 1-\varepsilon,  
\end{equation}
for some $\varepsilon > 0$. Sharper lower bounds within restricted classes of convex planar sets can be found in \cite[Theorems 1.4 \& 1.5]{VFNT2020thinning}.

In \cite{VBP21}, it is conjectured that the supremum of $F$ over $\mathcal{C}$ is approached by sequences of collapsing cones, while the infimum is approached by sequences of collapsing cylindroids.

On the other hand, the behavior of $G$ among $\mathcal{C}$ is different; indeed, one can prove that if $\om_n$ is a thinning sequence (defined in \autoref{thinset}), $G(\om_n)$ converges to $0$. Hence, a maximising sequence cannot collapse and we deduce the following proposition.
\begin{prop}\label{esistenza maX G}
    Let $G$ be the functional defined in \eqref{function_G} and $\mathcal{C}$ be the set defined in \eqref{open and convex}. Then

    \begin{equation*}
            \inf_{\om\in\mathcal{C}}G(\Omega) = 0,\qquad \text{and} \qquad
                    \sup_{\om\in \mathcal{C}}G(\Omega)  =\max_{\om\in \mathcal{C}}G(\Omega).  
    \end{equation*}
\end{prop}

In the planar case ($m=2$), we are able to show that any maximizer of $G$ over $\mathcal{C}$ enjoys additional regularity compared to other admissible sets.
\begin{prop}\label{regolarità}
Let $G$ be the functional defined in \eqref{function_G}, and let $\mathcal{C}$ be the set defined in \eqref{open and convex}. Let $m=2$ and let $\tilde{\Omega} \in \mathcal{C}$ satisfy
\[
G(\tilde{\Omega}) = \max_{\Omega \in \mathcal{C}} G(\Omega),
\]
then $\tilde{\Omega}$ is of class $C^1$.
\end{prop}

Whether the maximum in \autoref{esistenza maX G} is attained by a ball remains an open problem; numerical simulations seem to support this conjecture.
In this direction, we prove that the ball is a \textbf{local maximizer} of $G$. More precisely, we optimize the functional $G$ in the classes of \emph{nearly spherical domains}, defined as follows.

\begin{definizione}\label{def Sdg}
    Let $B_1$ 
denote the unit ball in 
$\R^m$. For all $h\in C^0(\partial B_1)$ such that $\norm{h}{L^\infty(\partial B_1)}<\frac{1}{2}$, we denote as $B_h$ the set defined through its boundary as
\[\partial B_h:=\set{x \in \mathbb{R}^m : x = (1 + h(y)) y, \, y \in \partial B_1}.\]
For any $\delta,\gamma>0$, we define the class $S_{\delta,\gamma}$ of nearly spherical sets as 
\begin{equation}
\label{sdg}
    S_{\delta,\gamma}:=\set{B_h: h\in C^{2,\gamma}(\partial B_1), \quad \int_{\partial B_1}h\,d\mathcal{H}^{m-1}=0,\quad \int_{\partial B_1}h x\,d\mathcal{H}^{m-1}=0\,\,\quad\text{and}\quad \norm{h}{C^{2,\gamma}}\leq \delta}.
\end{equation}

\end{definizione}
We now establish a local optimality and stability result for the ball within the class of nearly spherical sets.
\begin{teorema}\label{Optimality of the ball}
Let $G$ be the functional defined in \eqref{function_G}, and let $\gamma\in(0,1)$. There exists $\delta>0$ and a positive constant $C=C(m,\delta)$  such that for every $B_h\in S_{\delta,\gamma}$ (defined in \eqref{sdg})
\[G(B_1)-G(B_h)> C\norm{h}{H^1(\partial B_1)}^2.\]
   
\end{teorema}

The constant appearing in \autoref{Optimality of the ball} is not given explicitly.

Moreover, we use similar techniques to establish local optimality results for other Pólya-type functionals. 
In particular, for any \(q>0\), we consider the family of functionals
\begin{equation}
    \label{fq}
    F_q(\Omega) := \dfrac{T(\Omega)^q\,\Lambda(\Omega)}{\Vol(\om)^{\alpha_q}}, \qquad \text{where} \qquad \alpha_q = \dfrac{(m+2)q - 2}{m}.
\end{equation}
Notice that \(F_q\) coincides with the \emph{Kohler--Jobin functional} \cite{KJ1,KJ2} when \(q = \frac{2}{m+2}\), and with the \emph{Pólya functional} defined in \eqref{polyafunc} when \(q=1\). 

Several structural properties of the functionals $F_q$ have been established in the literature; see for instance \cite{VBP21,BBL24}. 

In particular, \eqref{bound polya} contains the bounds for the Pólya functional, 
while in \cite{KJ1} (see also \cite{B14}) the author proves that
\begin{equation}
\label{K-J}
    F_{\frac{2}{m+2}}(\Omega)
= T(\Omega)^{\frac{2}{m+2}} \Lambda(\Omega)
\;\ge\;
T(B)^{\frac{2}{m+2}} \Lambda(B) \qquad\forall \Omega \in \mathcal{O}.
\end{equation}
A summary of the known results regarding the infimum and supremum of $F_q$ over different classes of domains and ranges of $q$ is given in Table \ref{tab:Fq_properties}.

\begin{table}[H]
\centering
\renewcommand{\arraystretch}{1.3}
\begin{tabular}{|c|c|c|c|c|}
\hline
Set class & \( q \leq \frac{2}{m+2} \) & \( \frac{2}{m+2} < q < 1 \) & \( q = 1 \)& \( q > 1 \) \\
\hline
 \(\mathcal{O}\)
& \(\min_{\mathcal{O}} F_q = F_q(B_1)\) & \(\inf_{\mathcal{O}} F_q = 0\) &  \(\inf_{\mathcal{O}} F_q = 0\) & \(\inf_{\mathcal{O}} F_q = 0\) \\
& \(\sup_{\mathcal{O}} F_q = +\infty\) & \(\sup_{\mathcal{O}} F_q = +\infty\) & \(\sup_{\mathcal{O}} F_q  =1\) & \(\sup_{\mathcal{O}} F_q  <+\infty\) \\
\hline
\(\mathcal{C}\)
& \(\min_{\mathcal{C}} F_q = F_q(B_1)\) & \(\min_{\mathcal{C}} F_q > 0\) & \(\inf_{\mathcal{C}} F_q >0\)&\(\inf_{\mathcal{C}} F_q = 0\) \\
& \(\sup_{\mathcal{C}} F_q = +\infty\) & \(\sup_{\mathcal{C}} F_q = +\infty\) &  \(\sup_{\mathcal{C}} F_q < +\infty\)&\(\max_{\mathcal{C}} F_q < +\infty\) \\
\hline
\end{tabular}
\caption{Summary of infimum and supremum values of the functionals \(F_q\) over classes of domains for different ranges of \(q\).}
\label{tab:Fq_properties}
\end{table}

Furthermore, in \cite{BLNP23} it is shown that there exists $q_1 > 1$ such that for all $q \geq q_1$,
\[
\max_{\mathcal{C}} F_q = F_q(B_1).
\]

Denoting by $j_{\frac{m}{2}-1}$ the first positive zero of the Bessel function $J_{\frac{m}{2}-1}$, we have the following characterization of the local behaviour of $B_1$ for the functionals $F_q$.

\begin{teorema}\label{teor F_q}
Let $F_q$ be the functional defined in \eqref{fq} with $q>0$, and let $\gamma>0$. Then the following properties hold:
\begin{enumerate}[(i)]
 \item If $q<\dfrac{2}{m+2}$, then there exists $\delta>0$ and a constant $C=C(m,\delta, q)$  such that for every $B_h\in S_{\delta,\gamma}$ (defined in \ref{sdg})
\[F_q(B_h)-F_q(B_1)> C\norm{h}{H^{\frac{1}{2}}(\partial B_1)}^2.\]

\item  If $\dfrac{2}{m+2}<q<q^*:=\dfrac{2}{m+2}\left(\dfrac{j_{\frac{m}{2}-1}^2}{m}-1\right)$, $B_1$ is neither a local maximum nor a local minimum, i.e. for all $\delta>0$ there exist $B_{h_1},\,B_{h_2}\in S_{\delta,\gamma} $ such that 
\[F_q(B_{h_1})<F_q(B_1)\quad \text{and}\quad F_q(B_{h_2})>F_q(B_1). \]
\item If $q>q^* $, then there exists $\delta>0$ and a constant positive $C=C(m,\delta,q)$  such that for every $B_h\in S_{\delta,\gamma}$
\[F_q(B_1)-F_q(B_h)> C\norm{h}{H^{\frac{1}{2}}(\partial B_1)}^2.\]

\end{enumerate}
\end{teorema}

\begin{oss}\label{oss fq loc}
Since $2m < j^2_{\frac{m}{2}-1} < m \left( \frac{m}{2}+2 \right)$ (see \cite{ismail1995bounds}), it follows that $\dfrac{2}{m+2} < q^* < 1$. Consequently, in every dimension there exists an interval of $q$ values in which $B_1$ behaves as a saddle point. 
\end{oss}

\textbf{Organization of the Paper.} In \autoref{section2}, we recall some preliminary notions. In \autoref{section open sets}, we study the optimization problem over \(\mathcal{O}\). \autoref{optimization convex sets} is devoted to the existence and regularity of a maximizer for \(G\) over \(\mathcal{C}\). In \autoref{optimization nearly spherical}, we investigate the local properties of \(G\) and \(F_q\) within the class of nearly spherical domains. Finally, in \autoref{sec6}, we discuss some possible generalizations.

\section{Notation and Preliminaries}
\label{section2}  

This section is divided into two parts. In the first part, we recall the main definitions and properties of convex sets, while in the second part, we summarize the key results on shape derivatives that will be used throughout the paper.

\subsection{Some Properties of Convex Sets}

\begin{definizione}\label{support}
Let $\Omega$ be a bounded, open, and convex subset of $\mathbb{R}^m$. The \emph{support function} of $\Omega$ is defined as
\[
h_\Omega(y) = \sup_{x \in \Omega} (x \cdot y), \qquad y \in \mathbb{R}^m.
\]

Moreover, the \emph{width} of $\Omega$ in the direction $y \in \mathbb{R}^m$ is
\[
\omega_\Omega(y) = h_\Omega(y) + h_\Omega(-y),
\]
and, since the width is continuous with respect to the direction, there exists the \emph{minimal width} of $\Omega$ is
\[
w(\Omega) = \min \{ \omega_\Omega(y) \mid y \in \mathbb{S}^{m-1} \}.
\]

Moreover, we define the \emph{diameter} of $\Omega$ as 
$$\diam(\Omega) \coloneqq \sup \{ \lVert x - y \rVert_{\R^m} \mid x,y \in \Omega \}.
$$

\end{definizione}
\begin{definizione}
\label{thinset}
    Let $\{\Omega_n\}$ be a sequence of non-empty, bounded, open and convex sets of $\mathbb{R}^m$. We say that $\Omega_n$ is a sequence of thinning domains if
    \begin{equation*}
        \lim_n\dfrac{w({\Omega_n})}{\diam(\Omega_n)}=0.
    \end{equation*}
\end{definizione}
If we denote by $R_\Omega$ the inradius of $\Omega$, i.e.
\[
R_\Omega = \sup \{ r \in \mathbb{R} : B_r(x) \subset \Omega \text{ for some } x \in \Omega \},
\]
we have the following estimate, proved in \cite{fenchel_bonnesen} for the planar case and extended in \cite{brasco_2020_principal_frequencies} to arbitrary dimensions.

\begin{prop}\cite[Lemma B.1 ]{brasco_2020_principal_frequencies}
For any  $\Omega \in \mathcal{C}$, with $\mathcal{C}$ defined in \eqref{open and convex}, we have
\begin{equation} \label{measurepr}
\dfrac{1}{m} \leq \dfrac{\Vol(\Omega)}{P(\Omega) R_\Omega} < 1.
\end{equation}
The upper bound is sharp for a sequence of thinning cylinders, while the lower bound is sharp, for example, on balls. Moreover, for $m=2$, any circumscribed polygon (i.e., a polygon whose incircle touches all sides) attains the lower bound with equality. 
\end{prop}

Finally, we recall the following relations between the diameter, perimeter, volume, and minimal width of a convex set.

\begin{prop}\cite[Lemma 4.1]{EFT}
     For any  $\Omega \in \mathcal{C}$, with $\mathcal{C}$ defined in \eqref{open and convex}, then there exists a dimensional  constant $C=C(m)$ such that
     \begin{equation}
        \label{pmagrd}
    \diam(\Omega)\leq C \dfrac{P(\Omega)^{m-1}}{V(\Omega)^{m-2}}.
    \end{equation}
\end{prop}

\begin{prop}\cite[Equation (9)]{fenchel_bonnesen}
\label{prop:r>w}
Let $\Omega$ be a bounded, open and convex set of $\mathbb{R}^n$. Then
\begin{equation}
\label{eq:lowboundinradius}
    \displaystyle \frac{w(\Omega)}{2}\ge R_{\Omega}\ge\begin{cases}
    w({\Omega} )\displaystyle{\frac{\sqrt{n+2}}{2n+2}} & n \,\, \text{even}\\ \\
    w({\Omega}) \displaystyle{\frac{1}{2\sqrt{n}}} & n \,\, \text{odd}.
    \end{cases}
\end{equation}
\end{prop}

\subsection{Shape Derivatives and Local Optimality}
Let $\mathcal{A} \subset \mathcal{P}(\mathbb{R}^m)$ be a class of sets, and let $A:\mathcal{A} \to \mathbb{R}$ be a \emph{shape functional}, that is, a functional that assigns a real value to each set in $\mathcal{A}$. 
In the following, we recall some classical properties of shape derivatives, which describe the sensitivity of $A$ with respect to infinitesimal deformations of the set (see, for instance, \cite{HP18shape, DL19}).

\begin{teorema}(\textit{Structure Theorem of first and second shape derivatives}, \cite[Theorem 5.9.2]{HP18shape}). \label{Teorema struttura}
Let $\om\subset \R^m$ be an open, bounded and $C^3$ set. Let 
\[
\mathcal{V}(\Omega) =\set{(I+\theta)\om: \theta \in W^{1,\infty}(\R^m,\R^m), \,  \norma{\theta}_{W^{1,\infty}}\leq 1},
\]
and let $A$ be a shape functional defined on $\mathcal{V}(\om)$.
Consider the function 
\[A_\om:\theta\in\{\vartheta\in W^{1,\infty}(\RR^m,\R^m): \norm{\vartheta}{W^{1,\infty}}\leq 1\}\to A((I + \theta)(\Omega))\in\R,\]
and assume that $A_\om(\theta)$ is twice Fréchet-differentiable in $0$.
Then
\begin{enumerate}[(i)]
  \item[i.] there exists a continuous linear form $\ell_1[A](\om)$ on ${C}^1(\partial\Omega)$ such that 
  \[{A}_\Omega'(0)\xi = \ell_1[A](\om)(\xi|_{\partial\Omega} \cdot n)\] for all 
  $\xi \in {C}^\infty(\mathbb{R}^m, \mathbb{R}^m)$, where ${n}$ denotes the unit exterior normal vector on $\partial\Omega$.

  \item[ii.] there exists a continuous symmetric bilinear form $\ell_2[A](\om)$ on  ${C}^2(\partial\Omega)^2$ 
  such that for all $(\xi, \zeta) \in {C}^\infty(\mathbb{R}^m, \mathbb{R}^m)^2$
  \[
 {A}_\Omega''(0)(\xi, \zeta) = \ell_2[A](\om)(\xi \cdot n, \zeta \cdot n) 
  + \ell_1[A](\om)\big(\mathbf{B}(\zeta_\tau, \xi_\tau)) 
  - \nabla_\tau (\zeta \cdot n) \cdot \xi_\tau 
  - \nabla_\tau (\xi \cdot n) \cdot \zeta_\tau\big),
  \]
  where $\nabla_\tau$ is the tangential gradient, $\xi_\tau$ and $\zeta_\tau$ are the tangential components of $\xi$ and $\zeta$, and 
$\mathbf{B} = D_\tau n$ is the second fundamental form of $\partial\Omega$.
\end{enumerate}
When $\om$ is the unit ball we drop the dependence of the set.

\end{teorema}

\begin{oss} \label{oss calcolo l2}
Let $\om$ and $A$ be as in \autoref{Teorema struttura} and let $h \in C^\infty(\partial \Omega)$. Let $\om_t$ the set defined through its boundary as 
\[\partial\om_t= \{ x + t h(x) n(x) : x \in \partial \Omega \}.\]
If we consider $a:I\subset \R\to \R$ the function defined as
\[
a(t) := A\big( \om_t \big),
\]
then 
\[
a'(0) = \ell_1[A](\Omega) \cdot h, \quad \text{ and } \quad a''(0) = \ell_2[A](\Omega)(h,h).
\]
For further details, we refer to \cite[Section $5.9.4$]{HP18shape}.
\end{oss}
Following \cite{DL19}, we recall the definitions below, which make use of the fractional Sobolev space $H^s(\partial\Omega)$ 
, understood in the sense of \cite[Section~7.3]{lions}.

\begin{definizione} \label{defoptimal}
Let $\om$ and $A$ be as in \autoref{Teorema struttura}. \begin{enumerate}
    \item $\Omega^*:=\om$ is \emph{critical shape} for $A$ if
\[
\forall h \in C^\infty(\partial \Omega^*), \quad \ell_1[A](\Omega^*) (h) = 0.
\]
\item A critical shape $\Omega^*$ is a \emph{strictly stable shape} in $H^s(\partial \Omega^*)$ (under a volume constraint and up to translations) if:

\begin{enumerate}[(a)]
    \item $\ell_2[A](\Omega^*)$ extends continuously to $H^s(\partial \Omega^*)$;
\item Let \( h \in H^s(\partial\Omega^*) \setminus \{0\} \) satisfy
\begin{itemize}
    \item
    $\displaystyle \int_{\partial\Omega^*} h \, d\mathcal{H}^{m-1} = 0,$
    which ensures that the volume of the perturbed domain
    \( (I +  h n)\Omega^* \) is preserved at first order;
    \item $\displaystyle \int_{\partial\Omega^*} h\, x \, d\mathcal{H}^{m-1} = 0,
   $
    which ensures that the barycenter of
    \( (I + t h n)\Omega^* \) is preserved at first order.
\end{itemize}
Then
\[
\ell_2[A](\Omega^*)(h,h) > 0 .
\]

For notational convenience, we denote by
\( \mathcal{T}^s(\partial\Omega^*) \) the subspace of
\( H^s(\partial\Omega^*) \) just introduced:
\begin{equation}\label{def Ts}
\mathcal{T}^s(\partial\Omega^*)
:= \Big\{
\varphi \in H^s(\partial\Omega^*) :
\int_{\partial\Omega^*} \varphi \, d\mathcal{H}^{m-1} = \int_{\partial\Omega^*} \varphi\, x \, d\mathcal{H}^{m-1} = 0
\Big\}.
\end{equation}
\end{enumerate}
\end{enumerate}
\end{definizione}

Following \cite{DL19}, we also introduce the assumptions $(\mathbf{C}_{H^{s_2}})$ and $(\mathbf{IC}_{H^s, X})$, which will be used to provide sufficient conditions for local optimality.\\

\noindent \textbf{Assumption} \Cslabel{}: for $s\in (0,1]$, we say that the bilinear form $\ell$ acting on $\mathcal{C}^\infty(\partial\Omega)$ satisfies condition \Csref (and we say that $A$ satisfies the condition at $\Omega$ if $\ell_2[A](\Omega)$ does) if:
\[
 \text{there exist } s_1 \in [0, s) \text{ and } c_1 > 0 \text{ such that } \ell = \ell_m + \ell_r \text{ with}
\]
\[
\left\{
\begin{array}{ll}
\ell_m \text{ is lower semi-continuous in } H^{s}(\partial\Omega), \\
\ell_m(\varphi, \varphi) \geq c_1 |\varphi|^2_{H^{s}(\partial\Omega)}, \quad \forall \varphi \in \mathcal{C}^\infty(\partial\Omega), \\
\ell_r \text{ continuous in } H^{s_1}(\partial\Omega),
\end{array}
\right.
\]

where $|\cdot|_{H^{s}(\partial\Omega)}$ denotes the $H^{s}(\partial\Omega)$ semi-norm. In that case, $\ell$ is naturally extended (by density) to the space $H^{s}(\partial\Omega)$.

\noindent\textbf{Assumption }\IClabel{}: \quad 
Let $X$ be a Banach space with $C^\infty \subset X \subset W^{1,\infty}$ and $s\in(0,1]$. We say that $A$ satisfy assumption \ICref  at $\om$ if  there exist $\eta > 0$ and a modulus of continuity $\omega$ such that for every $h\in X$ with $\|h\|_X \leq \eta$, and all $t \in [0,1]$:
\[
|a''(t) - a''(0)| \leq \omega(\|h\|_X)\, \|h\|_{H^s}^2,
\]
where $a : t \in [0,1] \mapsto a(\Omega_t)\in\RR$ and $\om_t$ is such that
\[
\partial \Omega_t = \{ x + t h(x)\, n(x), \; x \in \partial \Omega \}.
\]

Using the assumptions and the definitions introduced, the following theorem holds.
\begin{teorema}\cite[Theorem 1.3]{DL19}
\label{Teorema DL}
Let $\Omega^*$ be of class $C^3$, and $A$ translation-invariant and twice Fréchet differentiable in a neighbourhood of $\Omega^*$ in $W^{1,\infty}$. Assume:
\begin{enumerate}[(a)]
    \item There exists $s \in (0,1]$ and a Banach space $X$ with $C^\infty \subset X \subset W^{1,\infty}$ such that $A$ satisfies assumptions \ICref and \Csref at $\Omega^*$;
    \item $\Omega^*$ is a critical shape for $A$;
    \item $\Omega^*$ is a \emph{strictly stable shape} in $H^s(\partial \Omega^*)$.
\end{enumerate}
Then there exist $\eta>0$ and $c>0$ such that for all $h \in \mathcal{T}^s(\partial\Omega^*)$ with $\|h\|_X \le \eta$,
\[
A(\Omega_h) - A(\Omega^*) \ge c \|h\|_{H^s(\partial\Omega^*)}^2.
\]
\end{teorema}

\section{Optimization in the class of open sets}\label{section open sets}
In this section, we establish sharp lower and upper bounds for the functional \( G \) within the class \(\mathcal{O}\) of open subsets of \(\mathbb{R}^m\) having finite perimeter, defined in \eqref{openperimeter}. The obtained results are summarized in \autoref{openopt}.

\begin{proof}[Proof of \autoref{openopt}]
   Let us prove that $\displaystyle\inf_{\om\in\mathcal{O}}G(\Omega)  = 0$. Consider $\set{\Omega_n}
   $ a sequence of sets such that 
   \[\lim_n \dfrac{\Vol({\Omega_n})}{P(\Omega_n)^{\frac{m}{m-1}}}=0.\]
  Using \eqref{bound F aperti}, we obtain that
\begin{equation}\label{to0}
   0 \leq G(\Omega_n) = F(\Omega_n) \dfrac{\Vol({\Omega_n})}{P(\Omega_n)^{\frac{m}{m-1}}} \leq \dfrac{\Vol({\Omega_n})}{P(\Omega_n)^{\frac{m}{m-1}}}.   
\end{equation}
Since the right-hand side converges to zero as $n$ approaches $+\infty$, we deduce that 
\[\inf_\mathcal{O} G=0.\]
Concerning the supremum of \( G \) over the class \(\mathcal{O}\), 
we first show that \( m^{\frac{m}{m-1}}\omega_m^{-\frac{1}{m-1}} \) provides an upper bound for \(G\). This follows directly from the upper bound in \eqref{bound F aperti} and the isoperimetric inequality. Indeed, recalling
\[
F(\Omega) < 1, \qquad 
\Vol(\Omega) P(\Omega)^{-\frac{m}{m-1}} \leq m^{\frac{m}{m-1}} \omega_m^{-\frac{1}{m-1}},
\]
we have that
\begin{equation}
    \label{gminiso}
G(\Omega) = F(\Omega)\Vol(\Omega)P(\Omega)^{-\frac{m}{m-1}} 
<\Vol(\Omega)P(\Omega)^{-\frac{m}{m-1}}  \leq m^{\frac{m}{m-1}}\omega_m^{-\frac{1}{m-1}}.
\end{equation}

We now prove that this bound is indeed the supremum of \( G \), by showing that for every open, bounded and regular set \( \Omega \), there exists a sequence of sets \( \{ \Omega_n \} \) such that \( G(\Omega_n) \) converges to the isoperimetric ratio of \( \Omega \), namely 
\[
\frac{\Vol(\Omega)}{P(\Omega)^{\frac{m}{m-1}}}.
\]

We use the homogenization construction used in \cite{DMM95} (see also \cite{BBP21}), by considering a capacitary measure \( \mu \), that is, a non-negative Borel measure absolutely continuous with respect to the capacitary measure, of the form
\[
\mu = c\,\mathcal{L}^m,
\]
where \( \mathcal{L}^m \) denotes the Lebesgue measure and \( c > 0 \) is a fixed constant.  

We fix \( n \in \mathbb{N} \) and tile the space \( \R^m \) with cubes \( \{ Q_i \} \) of side length \( 1/n \).  
We consider $B_i$ the ball of radius $1/2n$  having the same 
center as $Q_i$ and $E_i$ the ball concentric to $B_i$ such that 
\[
\operatorname{cap}(E_i, B_i) = \mu(Q_i) = \frac{c}{n^m}.
\]
We consider $\Omega_n:=\Omega\setminus \bigcup_i E_i$ (as in 
Figure   \ref{figura omogenizzazione}). Then we have (see \cite{DMM95} and  \cite{BBP21} for further details), 
\begin{equation*}
   \lim_n \Lambda(\Omega_n) =\Lambda(\Omega,\mu) \qquad
    \text{and} \qquad  \lim_n T(\Omega_n) = T(\Omega,\mu),
\end{equation*}
where
\begin{equation*}
    \Lambda(\Omega,\mu) = \min\left\{ \int_\Omega \abs{\nabla u}^2 \, dx + \int_\Omega u^2 \, d\mu: \, u \in W^{1,2}_0(\Omega) \cap L^2(\Omega,\mu), \int_\Omega u^2 =1\right\},
\end{equation*}
and
\begin{equation*}
    T(\Omega,\mu) = \max\left\{ \left(\int_\Omega u\, dx\right)^2\left(\int_\Omega \abs{\nabla u}^2 \, dx + c\int_\Omega u^2 \, dx\right)^{-1}: \, u \in W^{1,2}_0(\Omega) \cap L^2(\Omega,\mu)\right\}.
\end{equation*}
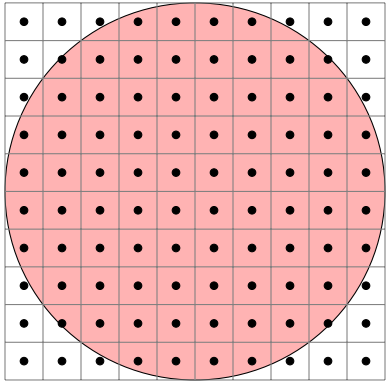
\begin{figure}[h]
    \centering
    \begin{tikzpicture}

        \fill[red!30] (2.5,2.5) circle(2.5cm);
        \draw (2.5,2.5) circle(2.5cm);
        \foreach \x in {0.25,0.75,...,4.75}
            \foreach \y in {0.25,0.75,...,4.75}
                \filldraw (\x,\y) circle(0.05cm);

                  \draw[step=0.5cm,gray,thin] (0,0) grid (5,5);
                 
    \end{tikzpicture}
        \caption{The set $\om_n$ in red}
        \label{figura omogenizzazione}
\end{figure}

By definition of $\Lambda(\Omega,\mu)$ and the fact that $\mu=c\mathcal{L}^m$, we have that \[\Lambda(\Omega,\mu)= c + \Lambda(\Omega).\]
Moreover, for $\delta > 0$, let us consider the function $u_\delta = \varphi(d(x,\partial \Omega))$, where
\[
\varphi(t)= 
\begin{cases}
    \dfrac{t}{\delta}, & \text{if } t < \delta,\\[6pt]
    1, & \text{if } t \geq \delta.
\end{cases}
\]
Denoting by $\Omega_\delta = \{ x \in \Omega : d(x, \partial \Omega) > \delta \}$ and using $u_\delta$ as test function in $T(\Omega,\mu)$, one can show that
\[
T(\Omega,\mu) \geq 
\frac{ \Vol(\Omega_\delta)^2 }{ \Vol(\Omega) } \,
\frac{1}{ \delta^{-2} + c }.
\]

Moreover, denoting by $r$ the radius of $E_i$, we have the following:
\begin{description}
    \item[Case $m = 2$]
$$\dfrac{c}{n^2}=\mu(Q_i)=cap(E_i,B_i)=\dfrac{2\pi }{\log\left(\dfrac{1}{2nr}\right)},$$

which gives

$$ r =\dfrac{1}{2n}e^{-\frac{2\pi n^2}{c}}.$$

    \item[Case $m\geq 3$]
    
$$\dfrac{c}{n^m}=\mu(Q_i)=cap(E_i,B_i)=\dfrac{m\omega_m(m-2)}{r^{2-m}-\left(\dfrac{1}{2n}\right)^{2-m}},$$
which gives

$$ r =\bigg(\dfrac{1}{n}\bigg)^{\frac{m}{m-2}} \left(\dfrac{c}{\omega_m(m-2)+ 2^{m-2}n^{-2}}\right)^{\frac{1}{m-2}} .$$
\end{description}

At this point we estimate the perimeter of the perforated set $\Omega_n$ in terms of the perimeter of the original domain and the perimeters of the removed balls. 
The perimeter of $\Omega_n$ can be controlled by the sum of the perimeter of $\Omega$ and the perimeter of the holes. 
More precisely, using the fact that each ball $E_i$ contributes at most its own perimeter, we obtain

$$P(\om_n)\leq P(\Omega)+\sum_{i : \, E_i \cap \Omega \ne \emptyset}P(E_i)\leq P(\Omega)+\sum_{i : \, Q_i \cap \Omega \ne \emptyset}P(E_i)= P(\Omega)+ C_n({\Omega})n^m m \omega_m r^{m-1}$$
where $C_n(\Omega) n^m$ denotes the number of cubes $Q_i$ that intersect $\Omega$. 
Since $C_n(\Omega)$ is a Riemann sum of the indicator function of $\Omega$ 
associated with the covering $\{Q_i : Q_i \cap \Omega \neq \emptyset\}$, 
it follows that $C_n(\Omega)$ converges to $\Vol(\Omega)$ as $n \to +\infty$.

Moreover, in any dimension $n^m r^{m-1}$ converges to $0$ as $n$  approaches $+\infty$, hence we get
\begin{equation}
\label{gmagiso}
    \liminf_{n \to \infty}G(\Omega_n) = \liminf_{n \to \infty}\dfrac{\Lambda(\Omega_n)T(\Omega_n)}{P(\Omega_n)^{\frac{m}{m-1}}} \geq \dfrac{\Lambda(\Omega,\mu)T(\Omega,\mu)}{P(\Omega)^{\frac{m}{m-1}}} \geq (c + \Lambda(\Omega))\dfrac{\Vol({\Omega_\delta})^2}{\Vol({\Omega})}\dfrac{1}{\delta^{-2} +c}\dfrac{1}{P(\Omega)^{\frac{m}{m-1}}}.
\end{equation}
By using \eqref{gminiso} and \eqref{relationgf}, and by  letting $c \to \infty$ and $\delta \to 0$, we obtain $$\frac{V(\Omega)}{P(\Omega)^{\frac{m}{m-1}}}\geq \limsup_{n \to \infty} G(\Omega_n)\geq \liminf_{n \to \infty} G(\Omega_n) \geq \frac{V(\Omega)}{P(\Omega)^{\frac{m}{m-1}}} .
$$

Choosing $\Omega = B_1$ we obtain the thesis. The supremum is not attained, since the inequality in \eqref{gminiso} is strict.
\end{proof}

\section{Optimization in the class of convex sets}\label{optimization convex sets}
In this section, we investigate the existence of a maximizer for the functional $G$ when the admissible class is restricted to $\mathcal{C}$ defined in \eqref{open and convex}. 
\begin{proof}[Proof of \autoref{esistenza maX G}]
In \autoref{openopt}, we showed that if $\{\Omega_n\}\subset \mathcal{O} $ is a sequence such that 
\[
\frac{\Vol(\Omega_n)}{P(\Omega_n)^{\frac{m}{m-1}}} \to 0,
\]
then $G(\Omega_n) \to 0$.

By combining inequalities \eqref{pmagrd}, \eqref{measurepr}, and \eqref{eq:lowboundinradius}, we obtain, for any convex set,
\begin{equation*}
\begin{split}
    \diam(\Omega_n)
    &\leq C(m) \frac{P(\Omega_n)^{m-1}}{\Vol(\Omega_n)^{m-2}} 
    = C(m) \frac{\Vol(\Omega_n)}{P(\Omega_n)} \frac{P(\Omega_n)^{m}}{\Vol(\Omega_n)^{m-1}}  \\
    &\leq C(m) R_{\Omega_n} \left(\frac{P(\Omega_n)^{\frac{m}{m-1}}}{\Vol(\Omega_n)} \right)^{m-1}
    \leq C'(m) w(\Omega_n) \left(\frac{P(\Omega_n)^{\frac{m}{m-1}}}{\Vol(\Omega_n)} \right)^{m-1}.
\end{split}
\end{equation*}

Hence,
\begin{equation}\label{vol  per thin}
    \frac{\Vol(\Omega_n)}{P(\Omega_n)^{\frac{m}{m-1}}} 
\leq C'(m) \left( \frac{w(\Omega_n)}{\diam(\Omega_n)}\right)^{\frac{1}{m-1}}.
\end{equation}

Therefore, if $\{\Omega_n\}\subset \mathcal{C}$ is a sequence of thinning sets (see \autoref{thinset}), then 
$$\lim_n G(\Omega_n) = 0.$$

Next step is to prove  the well-posedness of the maximization problem. 
To this aim, let $\{\Omega_n\}$ be  a maximizing sequence, such that $P(\Omega_n)=1$ $\forall n \in \mathbb{N}$. 
By \eqref{to0}, $\Vol(\om_n)$ necessary converges, up to a subsequence, to a positive constant $c$.

Inequality \eqref{pmagrd} gives us a uniform bound on the diameter of $\Omega_n$, i.e.
$$\diam(\Omega_n) \leq C(m) \dfrac{P(\Omega_n)^{m-1}}{\Vol({\Omega_n})^{m-2}} \leq \dfrac{2^{m-1}C(m)}{c^{m-1}},$$
hence $\Omega_n$, up to translations, is included in the ball of center 0 and radius $\dfrac{2^{m-1}C(m)}{c^{m-1}}.$
Therefore the sequence $\Omega_n$ has a Hausdorff limit $\tilde{\Omega}\in \mathcal{C}$ and
\begin{enumerate}
    \item $\displaystyle\lim_n\Lambda(\Omega_n) = \Lambda(\tilde{\Omega})$;
     \item $\displaystyle\lim_nT(\Omega_n) = T(\tilde{\Omega})$;
    \item $\displaystyle\lim_n P(\Omega_n) = P(\tilde{\Omega})$.
\end{enumerate}

Thus, from
$$\sup_{\Omega \in \mathcal{C}} G(\Omega)= \lim_n G(\Omega_n) = \lim_n \dfrac{\Lambda(\Omega_n)T(\Omega_n)}{P(\Omega_n)} =\dfrac{\Lambda(\tilde{\Omega})T(\tilde{\Omega})}{P(\tilde{\Omega})} \leq \sup_{\Omega \in \mathcal{C}} G(\Omega),$$
we get the thesis.
\end{proof}

In the following we show that, in the planar case \(m=2\), every maximizer of \(G\) within the class \(\mathcal C\) has a \(C^1\) boundary.

\begin{proof}[Proof of \autoref{regolarità}]
Let $\om\in\mathcal{C}$ be such that 
\[G(\om)=\max_{A\in \mathcal{C}}G(A).\]
Assume by contradiction that $\om$ is not of class $C^1$. 
Since \( \Omega \) is Lipschitz, this implies that \( \partial \Omega \) has at least one corner point. Without loss of generality, we can assume that the corner point is at the origin $O$, and we denote by $\alpha$ the angle of the corner and $\nu$ the versor parallel to its bisector.

We consider a similar construction as the one made in \cite[Lemma 3.3.2]{henrot2006extremum}, by considering, for all  $\varepsilon>0$, the set   
$C_\varepsilon=\{x\in\om: x\cdot \nu\leq \varepsilon\}$ and its complementary in $\om$, named $\om_\varepsilon=\om\setminus C_\varepsilon$.

Our aim is to reach a contradiction by proving that $G(\om_\varepsilon)>G(\om)$.

We start by estimating $T(\om_\varepsilon)$ from below.
Let \( w \) be the solution to the torsion problem \eqref{torsion problem}, then
\[T(\om) = \int_{\Omega} w\,dx.\]
Let \( \eta \in C^\infty_c(\mathbb{R}^2) \) be a cut-off function such that \( \eta = 0 \) in \( C_\varepsilon \), \( \eta = 1 \) outside \( C_{2\varepsilon} \),  \( Im(\eta)\subseteq [0,1] \) and  \( |\nabla \eta| \leq \dfrac{C}{\varepsilon} \) for some $C>0$.

Consider \( \widetilde{w} := w \eta\in H^1_0(\om_\varepsilon) \). Using $\widetilde{w}$ as a test function in the variational characterization of the torsion \eqref{variationa characteriz torsion}, we obtain
\begin{equation}\label{stima torsione}
   T(\Omega_\varepsilon) \geq \dfrac{\left( \int_{\Omega_\varepsilon} \widetilde{w}\,dx \right)^2}{\int_{\Omega_\varepsilon} |\nabla \widetilde{w}|^2\,dx}. 
\end{equation}

We recall that \( |\nabla w(0,0)| = 0 \) (for instance, see \cite{dauge2006elliptic}). Hence, If $\varepsilon$ is small enough 
\[
|w(x)| \leq  |x| \quad \text{in } C_{2\varepsilon}.
\]
Then for some constant $c_1>0$, 
\[
\int_{C_{2\varepsilon}} |w|\,dx \leq  c_1 \varepsilon^3.
\]
Thus, 
\[
T(\Omega) = \int_\Omega w\,dx = \int_{\Omega \setminus C_{2\varepsilon}} w \,dx+ \int_{C_{2\varepsilon}} w \,dx\leq \int_{\Omega_\varepsilon} \widetilde{w}\,dx +  c_1\varepsilon^3,
\]
and 
\[
\int_{\Omega_\varepsilon} \widetilde{w}\,dx \geq T(\Omega) - c_1 \varepsilon^3.
\]

Using \eqref{stima torsione}, we have
\[
T(\Omega_\varepsilon) \geq \dfrac{(T(\Omega) - c_1\varepsilon^3)^2}{\int_{\Omega_\varepsilon} |\nabla \widetilde{w}|^2\,dx}.
\]

We estimate the denominator:
\[
\int_{\Omega_\varepsilon} |\nabla \widetilde{w}|^2\,dx \leq \int_\Omega |\nabla w|^2\,dx + \int_{C_{2\varepsilon}} (|\nabla \eta|^2 {w}^2 +2\abs{\nabla w}\abs{ \nabla \eta} w)\,dx.
\]
Since \( |\nabla \eta|  \leq \dfrac{C}{\varepsilon} \) and \( \abs{\nabla w} \) is bounded, we have
\begin{equation*}
    \begin{aligned}
        \int_{C_{2\varepsilon}} (|\nabla \eta|^2 {w}^2 +2\abs{\nabla w}\abs{ \nabla \eta} w)\,dx &\leq \dfrac{C}{\varepsilon^2} \int_{C_{2\varepsilon}} w^2 \,dx  + \dfrac{C}{\varepsilon} \int_{C_{2\varepsilon}} w \,dx \\&\leq \dfrac{C }{\varepsilon^2} \int_{C_{2\varepsilon}} |x|^2\,dx + \dfrac{C}{\varepsilon} \int_{C_{2\varepsilon}} |x||\nabla w|\,dx \leq c_2 \varepsilon^2.
    \end{aligned}
\end{equation*}

Hence, it holds 
\begin{equation}\label{stima torisione1}
    T(\Omega_\varepsilon) \geq \dfrac{(T(\Omega) - c_1  \varepsilon^3)^2}{T(\Omega) + c_2 \varepsilon^2}.
\end{equation}

Now we want to estimate from above the perimeter of $\Omega_\varepsilon$. Firstly, we observe that $\partial C_\varepsilon$ intersects $\partial \Omega$ at the origin and in two other points $A$ and $B$. 

Then we construct the triangle $T_\varepsilon = \triangle OAB \subset \Omega$ and the point $M$ that is the intersection between the side $\overline{AB}$ and the bisector of $\alpha$, see \autoref{figura regolarità}. Let 
\[
\beta = \angle AOB, \qquad \beta_1 = \angle AOM, \qquad \beta_2 = \angle MOB,
\] 
so that
\[
\beta \leq \alpha, \qquad \beta_1 \leq \frac{\alpha}{2}, \qquad \beta_2 \leq \frac{\alpha}{2}.
\]
Moreover, the lengths of the sides and segments are given by
\[
\overline{OA} = a = \frac{\varepsilon}{\cos \beta_1}, \qquad
\overline{OB} = b = \frac{\varepsilon}{\cos \beta_2}, \qquad
\overline{AB} = c = \varepsilon (\tan \beta_1 + \tan \beta_2), \qquad
\overline{OM} = \varepsilon.
\]

\tikzset{every picture/.style={line width=0.75pt}} 
\begin{figure}[H]
    \begin{center}
\begin{tikzpicture}[x=0.75pt,y=0.7pt,yscale=-0.7,xscale=0.9]

\draw  [color={rgb, 255:red, 208; green, 2; blue, 27 }  ,draw opacity=1 ] (327,35) -- (596.5,306.5) -- (57,306.5) -- cycle ;
\draw    (141,335) .. controls (147,323) and (282.5,80.5) .. (326.4,35.2) ;
\draw    (327,35) .. controls (356.67,61) and (465,344) .. (468,355) ;
\draw  [dash pattern={on 0.84pt off 2.51pt}]  (327,35) -- (328,306) ;
\draw  [color={rgb, 255:red, 74; green, 144; blue, 226 }  ,draw opacity=1 ] (327,35) -- (449,306.5) -- (157,306.5) -- cycle ;
\draw  [draw opacity=0] (355.11,98.34) .. controls (332.47,105.41) and (309.25,103.24) .. (290.12,93.93) -- (331.7,22.41) -- cycle ; \draw  [color={rgb, 255:red, 74; green, 144; blue, 226 }  ,draw opacity=1 ] (355.11,98.34) .. controls (332.47,105.41) and (309.25,103.24) .. (290.12,93.93) ;  
\draw  [draw opacity=0] (358.54,67.03) .. controls (348.49,72.3) and (337.69,75.23) .. (326.41,75.35) .. controls (314.97,75.48) and (303.96,72.72) .. (293.69,67.54) -- (324.77,-71.28) -- cycle ; \draw  [color={rgb, 255:red, 208; green, 2; blue, 27 }  ,draw opacity=1 ] (358.54,67.03) .. controls (348.49,72.3) and (337.69,75.23) .. (326.41,75.35) .. controls (314.97,75.48) and (303.96,72.72) .. (293.69,67.54) ;  
\draw  [draw opacity=0] (327.55,128.81) .. controls (327.45,128.81) and (327.34,128.81) .. (327.23,128.81) .. controls (300.61,128.99) and (279.82,116.45) .. (266.7,95.34) -- (371.85,-25.46) -- cycle ; \draw  [color={rgb, 255:red, 208; green, 2; blue, 27 }  ,draw opacity=1 ] (327.55,128.81) .. controls (327.45,128.81) and (327.34,128.81) .. (327.23,128.81) .. controls (300.61,128.99) and (279.82,116.45) .. (266.7,95.34) ;  
\draw  [draw opacity=0] (371.88,136.91) .. controls (355.9,143.43) and (340.05,144.29) .. (326.31,140.04) -- (362.7,53.06) -- cycle ; \draw  [color={rgb, 255:red, 74; green, 144; blue, 226 }  ,draw opacity=1 ] (371.88,136.91) .. controls (355.9,143.43) and (340.05,144.29) .. (326.31,140.04) ;  
\draw  [draw opacity=0] (326.07,161.66) .. controls (302.14,164.96) and (277.75,159.38) .. (257.78,147.64) -- (302.7,86.06) -- cycle ; \draw  [color={rgb, 255:red, 74; green, 144; blue, 226 }  ,draw opacity=1 ] (326.07,161.66) .. controls (302.14,164.96) and (277.75,159.38) .. (257.78,147.64) ;  

\draw (234,185.15) node [anchor=north west][inner sep=0.75pt]    {$\textcolor[rgb]{0.29,0.56,0.89}{a}$};
\draw (383,185.15) node [anchor=north west][inner sep=0.75pt]    {$\textcolor[rgb]{0.29,0.56,0.89}{b}$};
\draw (345,293.15) node [anchor=north west][inner sep=0.75pt]    {$\textcolor[rgb]{0.29,0.56,0.89}{c}$};
\draw (311,250.15) node [anchor=north west][inner sep=0.75pt]    {$\varepsilon $};
\draw (312.75,75.4) node [anchor=north west][inner sep=0.75pt]    {$\textcolor[rgb]{0.82,0.01,0.11}{\alpha }$};
\draw (315.25,103.1) node [anchor=north west][inner sep=0.75pt]    {$\textcolor[rgb]{0.29,0.56,0.89}{\beta }$};
\draw (277.25,116.9) node [anchor=north west][inner sep=0.75pt]    {$\textcolor[rgb]{0.82,0.01,0.11}{{\textstyle \dfrac{\alpha }{2}}}$};
\draw (291.25,162.4) node [anchor=north west][inner sep=0.75pt]    {$\textcolor[rgb]{0.29,0.56,0.89}{\beta }\textcolor[rgb]{0.29,0.56,0.89}{_{1}}$};
\draw (346.75,144.4) node [anchor=north west][inner sep=0.75pt]  [color={rgb, 255:red, 74; green, 144; blue, 226 }  ,opacity=1 ]  {$\beta _{2}$};
\draw (340,26.9) node [anchor=north west][inner sep=0.75pt]    {$O=(0,0)$};
\draw (157.25,306.9) node [anchor=north west][inner sep=0.75pt]    {$A$};
\draw (455.25,306.9) node [anchor=north west][inner sep=0.75pt]    {$B$};
\draw (320,306.9) node [anchor=north west][inner sep=0.75pt]    {$M$};

\end{tikzpicture}
 \caption{$T_\varepsilon$ in blue}
    \label{figura regolarità}
     \end{center}
\end{figure}
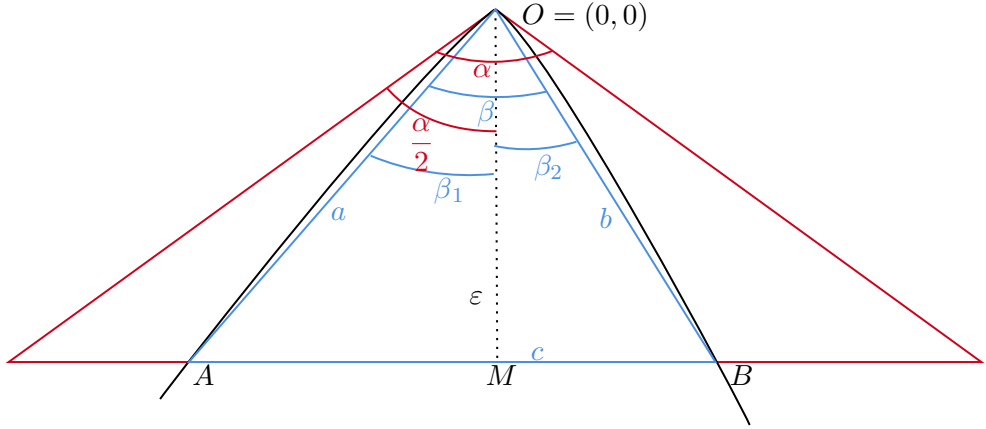

Hence\[P(\om)-P(\om_\varepsilon)\geq a+b-c=\varepsilon\left[\dfrac{1}{\cos(\beta_1)}(1-\sin(\beta_1)+\dfrac{1}{\cos(\beta_2)}(1-\sin(\beta_2))\right]\]
\[\geq \varepsilon(2-\sin(\beta_1)-\sin(\beta_2))=2\varepsilon\left[1-\sin\left(\dfrac{\beta_1+\beta_2}{2}\right)\cos\left(\dfrac{\beta_1-\beta_2}{2}\right)\right]\geq 2\varepsilon\left[1-\sin\left(\dfrac{\beta}{2}\right)\right].\]
Therefore, we have that 
\begin{equation}\label{stima perimetro}
    P(\om_\varepsilon)\leq P(\om)\left(1-\dfrac{2\varepsilon}{P(\om)}\left(1-\sin\left(\dfrac{\beta}{2}\right)\right)\right).
\end{equation}

Then, using the monotonicity of $\Lambda$, \eqref{stima torisione1} and \eqref{stima perimetro}, we deduce 
\[
G(\Omega_\varepsilon) \geq \Lambda(\Omega) \dfrac{T(\Omega_\varepsilon)}{P(\Omega_\varepsilon)^2} \geq \Lambda(\Omega) \dfrac{T(\Omega)}{P(\Omega)^2}  \dfrac{\left(1 - \dfrac{c_1}{T(\Omega)}  \varepsilon^3\right)^2}{\left(1 + \dfrac{c_2 \varepsilon^2}{T(\Omega)}\right) \left(1 - \dfrac{2\varepsilon}{P(\Omega)}\left(1 - \sin\left( \dfrac{\beta}{2} \right)\right) \right)^2}.
\]
Using Bernulli inequality, we obtain 
\[G(\om_\varepsilon)\geq G(\om)\left(1 - \dfrac{c_1}{T(\Omega)}  \varepsilon^3\right)^2\left(1 - \dfrac{c_2  \varepsilon^2}{T(\Omega)}\right) \left(1 + \dfrac{4\varepsilon}{P(\Omega)}\left(1 - \sin\left( \dfrac{\beta}{2} \right)\right) \right),\]

Thus, we have that

\[
G(\Omega_\varepsilon) \geq G(\Omega) \left( 1 + \dfrac{4\varepsilon}{P(\Omega)} \left(1 - \sin\left( \dfrac{\beta}{2} \right) \right) + o(\varepsilon) \right).
\]
Hence, for \( \varepsilon \) small enough, \( G(\Omega_\varepsilon) > G(\Omega) \), contradicting the maximality of \( \Omega \). This contradiction implies that \( \partial \Omega \) cannot have a corner point, thus \( \Omega \) is of class \( C^1 \).

\end{proof}

\section{Optimality of the ball for nearly spherical domains}\label{optimization nearly spherical}

In this section, we prove \autoref{Optimality of the ball} and \autoref{teor F_q}. 
To this end, we show that both $F_q$ and $G$ satisfy the assumptions of \autoref{Teorema DL} when evaluated on the ball. 
Once this is established, the desired results will follow directly. 

Using the notation introduced in \autoref{Teorema struttura} and setting $X = C^{2,\gamma}$ with $\gamma>0$, it is well known (see, for instance, \cite[Chapter 5]{HP18shape}) that 
$T_{B_1}$, $\Lambda_{B_1}$, $P_{B_1}$, and $\Vol_{B_1}$ are twice differentiable at $0$. 
As a consequence, the same differentiability property holds for ${F_q}_{B_1}$ and $G_{B_1}$.

\subsection{On the improved continuity condition}

Let $\Omega \subset \mathbb{R}^m$ be a fixed domain of class $C^3$. 
The purpose of this section is to show that the functionals $F_q$ 
satisfy assumption \ICref{} with $s = \tfrac{1}{2}$ and 
$X = C^{2,\gamma}(\partial \Omega)$, with $\gamma > 0$, 
while $G$ satisfies \ICref{} with $s = 1$ and the same choice of $X$.

Since $F_q$ and $G$ are defined in terms of functionals that verify assumption 
\ICref  for suitable $s$ and $X$, it suffices to prove that, 
under appropriate conditions, functionals constructed in this way 
inherit the same property.

We stress that, in order to apply \autoref{Teorema DL} at the end, we will only need to consider 
the case $\Omega = B_1$. Therefore, the assumption $\Omega \in C^3$ is not restrictive.

\begin{lemma}\label{lemma der prime ic}
Let $\om\subset \R^m$ be an open and bounded set. Let 
\[
\mathcal{V}(\Omega) =\set{(I+\theta)\om\in W^{1,\infty}(\RR^m,\R^m): \norma{\theta}_{W^{1,\infty}}<1},
\]
and let $A$ be a shape functional defined on $\mathcal{V}(\om)$.

Moreover, let $X \subset Y \subset W^{1,\infty}(\partial \Omega)$ be two normed spaces, and assume that the immersion
\[
X \hookrightarrow Y
\]
is continuous.

Assume that there exists $\eta > 0$ such that, for all $h \in X$ with $\|h\|_X \leq \eta$ and for all $t \in [0,1]$, the perturbed domain
\[
\partial \Omega_t = \{\, x + t\, h(x)\, n(x) \; : \; x \in \partial \Omega \,\}
\]
is well-defined. We then set
\[
a(t) = A(\Omega_t).
\]

Moreover, we assume that the following properties hold.
\begin{enumerate}
    \item $a$ is twice differentiable on $[0,1]$;
    \item there exists a constant $c > 0$ such that 
    \[
    |a'(0)| \leq c\, \|h\|_{Y};
    \]
    \item there exists a constant $c > 0$ such that 
    \[
    |a''(0)| \leq c\, \|h\|_{Y}^2;
    \]
    \item there exists a modulus of continuity $\omega$ such that
    \[
    |a''(t) - a''(0)| 
    \leq \omega(\|h\|_X)\, \|h\|_{Y}^2,
    \qquad \forall t \in [0,1].
    \]
\end{enumerate}

Then there exist a modulus of continuity $\omega$ and a constant $k > 0$ such that, for every $t \in [0,1]$,
\[
|a'(t) - a'(0)| \leq \omega(\|h\|_X)\, \|h\|_{Y}, 
\qquad
|a'(t)| \leq k\, \|h\|_{Y},
\qquad
|a(t) - a(0)| \leq k\, \|h\|_{Y}.
\]

\end{lemma}
\begin{proof}
Fix $t\in[0,1]$, then there exists $\xi_t\in (0,t)$ such that 
    \[|a'(t)-a'(0)|\leq |a''(\xi_t)|\leq |a''(\xi_t)-a''(0)|+|a''(0)|\leq \omega(\norm{h}{X})\norm{h}{Y}^2+c\norm{h}{Y}^2=\omega_1(\norm{h}{X} )\norm{h}{Y}.\]
    Moreover, we have that 
    \[|a'(t)|\leq |a'(t)-a'(0)|+|a'(0)|\leq k \norm{h}{Y}.\]
    Finally, for all $t\in[0,1]$ there exists $\zeta_t\in(0,t)$ such that 
    \[|a(t)-a(0)|\leq |a'(\zeta_t)|\le k\norm{h}{Y}.\]
\end{proof}

\begin{lemma}\label{lemma con f}
Let $\Omega$, $\mathcal{V}(\Omega)$, $X$, $Y$ $\eta$ and $\om_t$ be as in the hypothesis of \autoref{lemma der prime ic}, and let 
$\{A_i\}_{i=1}^n$ be a family of shape functionals defined on $\mathcal{V}(\Omega)$.

We set
\[
a_i(t) = A_i(\Omega_t), \qquad i = 1, \dots, n.
\]
Moreover, we assume that the following properties hold for every $i = 1, \dots, n$:
\begin{enumerate}
    \item $a_i$ is twice differentiable on $[0,1]$;
    \item there exists a constant $c > 0$ such that 
    \[
    |a_i'(0)| \leq c\, \|h\|_{Y};
    \]
    \item there exists a constant $c > 0$ such that 
    \[
    |a_i''(0)| \leq c\, \|h\|_{Y}^2;
    \]
    \item there exists a modulus of continuity $\omega$ such that
    \[
    |a_i''(t) - a_i''(0)| 
    \leq \omega(\|h\|_X)\, \|h\|_{Y}^2,
    \qquad \forall t \in [0,1].
    \]
\end{enumerate}

\medskip

Let $U \subset \mathbb{R}^n$ be an open set such that 
\[
(A_1(O), \dots, A_n(O)) \in U 
\quad \text{for all } O \in \mathcal{V}(\Omega),
\]
and let $f \in C^{2,\alpha}(U)$ with $\alpha > 0$.  
Define the functional
\[
A : \mathcal{V}(\Omega) \to \mathbb{R},
\qquad
A(O) = f\bigl(A_1(O), \dots, A_n(O)\bigr).
\]

Then the function 
\[
a : [0,1] \to \mathbb{R}, 
\qquad 
a(t) = A(\Omega_t),
\]
satisfies properties (1)–(4) above.
\end{lemma}
\begin{proof}
    Since $a(t)=f(a_1(t),\dots, a_n(t))$, property $(1)$ follows from the regularity assumptions on $f$. By computing the derivatives, we obtain  
    \[a'(t)=\sum_{i=1}^n \derivparzialecompl{f}{x_i}(a_1(t),\dots, a_n(t))\,a'_i(t),\]
    \[a''(t)=\sum_{i,k=1}^n\derivparzialecompl{^2f}{x_i\partial x_k}(a_1(t),\dots, a_n(t))\, a'_i(t)a'_k(t)\;+\sum_{i=1}^n \derivparzialecompl{f}{x_i}(a_1(t),\dots, a_n(t))\,a''_i(t).\]
Since, for every $i \in \{1, \dots, n\}$, the function $a_i$ satisfies properties (2) and (3), it follows that the same estimates hold for $a$.

We now want to prove that $a$ satisfy property (4). By direct computation, we have that
\begin{equation}\label{dis comp ICj}
\begin{split}
    |a''(t)-a''(0)|&\leq \sum_{i,k=1}^n\left|\derivparzialecompl{^2f}{x_i\partial x_k}(a_1(t),\dots, a_n(t))\, a'_i(t)a'_k(t)-\derivparzialecompl{^2f}{x_i\partial x_k}(a_1(0),\dots, a_n(0))\, a'_i(0)a'_k(0)\right|\\
    &+\sum_{i=1}^n \left|\derivparzialecompl{f}{x_i}(a_1(t),\dots, a_n(t))\,a''_i(t) -\derivparzialecompl{f}{x_i}(a_1(0),\dots, a_n(0))\,a''_i(0)\right|\\
&=:\sum_{i,k=1}^nR_{i,k}+\sum_{i=1}^n S_i.
    \end{split}
\end{equation}
For all $i=1,\dots.n$, we have that 
\[S_i\leq |a''_i(t)-a''_i(0)|\left|\derivparzialecompl{f}{x_i}(a_1(t),\dots, a_n(t))\right|+|a''_i(0)|\left|\derivparzialecompl{f}{x_i}(a_1(t),\dots, a_n(t))-\derivparzialecompl{f}{x_i}(a_1(0),\dots, a_n(0))\right|.\]
Hence, by the lipschitzianity of the first derivatives of $f$ and \autoref{lemma der prime ic}, we deduce 
\[\left|\derivparzialecompl{f}{x_i}(a_1(t),\dots, a_n(t))-\derivparzialecompl{f}{x_i}(a_1(0),\dots, a_n(0))\right|=\omega(\norm{h}{X})\]
As a consequence, by assumptions (3) and (4), we obtain 
\begin{equation}\label{dis si}
    S_i\leq\omega(\|h\|_X)\, \|h\|_{Y}^2,
\end{equation} 
We can repeat a similar argument in order to bound $R_{i,k}$ from above, indeed
\begin{equation*}
\begin{split}
     R_{i,k}&\leq|a_i'(t)-a'_i(0)||a'_k(t)|\left|\derivparzialecompl{^2f}{x_i\partial x_k}(a_1(t),\dots ,a_n(t))\right|+|a'_i(0)|\left|\derivparzialecompl{^2f}{x_i\partial x_k}(a_1(t),\dots ,a_n(t))\right||a'_k(t)-a'_k(0)|\\&+|a'_i(0)||a'_k(0)|\left|\derivparzialecompl{^2f}{x_i\partial x_k}(a_1(t),\dots ,a_n(t))-\derivparzialecompl{^2f}{x_i\partial x_k}(a_1(0),\dots ,a_n(0))\right|.
    \end{split}
\end{equation*}
Using \autoref{lemma der prime ic} and the H\"older continuity and the boundedness of the second derivatives of $f$, we obtain  
\begin{equation}\label{dis R_i}
    R_{i,k}\leq \omega(\norm{h}{X})\norm{h}{Y}^2.
\end{equation}
Using \eqref{dis comp ICj}, \eqref{dis R_i} and \eqref{dis si} we obtain that 
\[|a''(t)-a''(0)|\leq \omega(\norm{h}{X})\norm{h}{Y}^2.\]

\end{proof}

\begin{oss}\label{corollario IC}
In the case \( Y = H^{s}(\partial \Omega) \), properties (1) and (4) coincide with assumption \ICref.  
Therefore, by \autoref{lemma con f}, if properties (2), (3), and \ICref{} hold for each \(A_i\), then they also hold for the functional \(A\).  
In particular, \(A\) satisfies assumption \ICref{} at \(\Omega\).

\end{oss}
\begin{lemma}\label{lemma prop lambda etc}
    Let $\om\subseteq \RR^m$ be a bounded $C^3$ open set and $\gamma>0$. The following properties hold:
    \begin{enumerate}[(i)]
        \item $\Lambda$ and $T$ verify the assumptions of \autoref{lemma con f} for $Y=H^{\frac{1}{2}}(\partial \om)$ and $X=C^{2,\gamma}(\partial \om)$,
        \item $P$ verifies the assumptions of \autoref{lemma con f} for $Y=H^{1}(\partial \om)$ and $X=C^{2,\gamma}(\partial \om)$,
        \item $V$ verifies the assumptions of \autoref{lemma con f} for $Y=L^2(\partial \om)$ and $X=C^{2,\gamma}(\partial \om)$.
    \end{enumerate}
\end{lemma}
\autoref{lemma prop lambda etc} was proven in \cite{dambrine2002variations} and later in a more general setting in \cite{DL19}.

\begin{prop}\label{proposizione IC}
Let $\om\subset\RR^m$ be $C^3$ open and bounded, let $F_q$ be the functional defined in \eqref{fq}, with $q>0$ and $G$ the functional defined in \eqref{function_G}.
 $F_q$ satisfies the assumption \ICref at $\om$ for $s=\frac{1}{2}$ and $X=C^{2,\gamma}(\partial \om),$ for $\gamma>0$. Moreover, $G$ satisfies the assumption \ICref at $\om$ for the same choice of $X$ and $s=1$.
\end{prop}
\begin{proof}
    The result directly follows from  \autoref{corollario IC} and \autoref{lemma prop lambda etc}.
\end{proof}

\subsection{Criticality of the ball}

In this section, we compute the first derivatives of $F_q$ and $G$ at the ball (in the sense of \autoref{Teorema struttura}) and prove that they vanish. 

\begin{lemma}\label{lemma shape critica}
Let $B_1 \subset \mathbb{R}^m$ denote the unit ball. 
Consider the functional $F_q$ defined in \eqref{fq}, with $q > 0$, 
and the functional $G$ defined in \eqref{function_G}. 
Then $B_1$ is a critical shape for $F_q$ and $G$ 
(in the sense of \autoref{defoptimal}).
\end{lemma}
\begin{proof}
Since $F_q$ and $G$ are scale invariant, it is sufficient to prove that if $\varphi$ has zero mean on $\partial \om$, then $\ell_1[F_q](\varphi)=\ell_1[G](\varphi)=0$.\\
    Using Leibniz's rule, we have that 
    \[\ell_1[F_q]\varphi=\dfrac{T(B_1)^{q-1}}{\Vol(B_1)}\left[\ell_1[\Lambda](\varphi) T(B_1)+qT(B_1)\Lambda(B_1)\ell_1[T](\varphi)-\alpha_q\dfrac{\Lambda(B_1)T(B_1)\ell_1[Vol](\varphi)}{\Vol(B_1)}\right],\]
   and
    \[\ell_1[G]= \dfrac{1}{P(B_1)^{\frac{m}{m-1}}}\left[\ell_1[\Lambda](\varphi) T(B_1)+\Lambda(B_1)\ell_1[T](\varphi)-\dfrac{\Lambda(B_1)T(B_1)\ell_1[P](\varphi)}{P(B_1)}\right].\]
    Since the derivatives of $T,\Lambda,P$ and $\Vol$ vanish over the ball, we obtain 
     \[\ell_1[F_q](\varphi)=0\quad \ell_1[G](\varphi)= 0.\]
\end{proof}

\subsection{On the coercivity conditions}

In this section, we investigate the coercivity properties of the bilinear forms corresponding to the second derivatives of the functionals $F_q$ and $G$ (defined in \autoref{Teorema struttura}).

To this end, we make use of spherical harmonics.  
Let \(\mathcal{H}_k\) denote the space of spherical harmonics of degree \(k\), and let \(\{Y^{k,\ell}\}_{1 \leq \ell \leq N_k}\) be an orthonormal basis of \(\mathcal{H}_k\) with respect to the \(L^2(\partial B_1)\) scalar product.
Then the collection \(\{Y^{k,\ell}\}_{k \in \mathbb{N},\, 1 \leq l \leq N_k}\) forms a Hilbert basis for \(L^2(\partial B_1)\), and any function \(\varphi \in L^2(\partial B_1)\) can be expanded as:
\begin{equation}\label{espansione phi}
    \varphi(x) = \sum_{k=0}^\infty \sum_{l=1}^{N_k} \varphi_{k,l} \, Y^{k,l}(x), \quad \text{for } |x| = 1.
\end{equation}

Furthermore, the sequence \((\varphi_{k,l})\) characterizes the Sobolev regularity of \(\varphi\): we have \(\varphi \in H^s(\partial B_1)\) if and only if the series
\[
\sum_{k=0}^\infty (1 + k^2)^s \sum_{\ell=1}^{N_k} |\varphi_{k,\ell}|^2
\]
is convergent.
\begin{oss}\label{oss sferiche Ts} Let $\mathcal{T}^0(\partial B_1)$ be the functional space defined in \eqref{def Ts}.
    If $\varphi\in L^2(\partial B_1)$,  then $\varphi\in \mathcal {T}^0(\partial B_1)$ if and only if $\varphi_{k,l}=0$ for $k=0,1$. 

\end{oss}

\begin{definizione}\label{notazione derivate seconde}
   Let $\varphi\in C^{2,\gamma}(\partial B_1)$, with $\norma{\varphi}_{L^\infty(\partial B_1)}<\frac{1}{2}$. Let $t\in [0,1]$ and consider $B_t\varphi$ the set defined through its boundary
   \[\partial B_{t\varphi}=\{x+t\varphi(x)x:x\in \partial B_1\}.\]
   We define the following functions
   \[p:t\in[0,1]\to P(B_{t\varphi})\in\RR, \qquad \vol:t\in[0,1]\to V(B_{t\varphi})\in\RR,\]
\[\tau:t\in[0,1]\to T(B_{t\varphi})\in\RR,\qquad \lambda:t\in[0,1]\to \Lambda(B_{t\varphi})\in\RR.\]
\end{definizione}

It is well known (see, for instance, \cite[Chapter 5]{HP18shape} or \cite{DL19}) that if 
\[
\int_{\partial B_1} \varphi \, d\mathcal{H}^{m-1} = 0,
\]
then the first derivatives of all the functions defined in \autoref{notazione derivate seconde} vanish at \(0\).  
Furthermore, in the following lemma we recall the expressions of their second derivatives.

\begin{lemma} \label{derivate seconde} Let $\varphi\in C^{2,\gamma}(\partial B_1)$, with $\gamma>0$. Using the notation introduced in \autoref{notazione derivate seconde} and the expansion of the function $\varphi$ in \eqref{espansione phi}, it holds that 

   \begin{align*}
        \vol''(0)
    &= \sum_{k=0}^{\infty} \sum_{l=1}^{d_k} (m-1)\, \varphi_{k,l}^2, \\[1ex]
\tau''(0)
    &= -\sum_{k=0}^{\infty} \sum_{l=1}^{d_k} \left( \dfrac{2}{m^2}k - \dfrac{m+1}{m^2} \right) \varphi_{k,l}^2, \\[1ex]
p''(0) 
    &= \sum_{k=0}^{\infty} \sum_{l=1}^{d_k} \left( k^2 + (m-2)k + (m-1)(m-2) \right) \varphi_{k,l}^2, \\[1ex]
\lambda''(0)
    &= \beta_m^2 \left( 
        3 \varphi_{0,1}^2 
        + \sum_{k=1}^{\infty} \sum_{l=1}^{d_k} 
        2 \left[ 
            k + \dfrac{m-1}{2} 
            - c_k
        \right] \varphi_{k,l}^2 
    \right),
    \end{align*}
    where 
    \[\beta_m=\dfrac{2\Lambda(B_1)}{P(B_1)},
    \]
    and
    \[ c_k=j_{\frac{m}{2}-1} \dfrac{J_{k+\frac{m}{2}}(j_{\frac{m}{2}-1})}{J_{k-1+\frac{m}{2}}(j_{\frac{m}{2}-1})}. \]
    
\end{lemma}
\begin{oss}\label{oss ck} We recall the recurrence formula for the Bessel functions (see \cite[Chapter 9]{abramowitz1948handbook})
\[\dfrac{2\nu}{z}J_\nu(z)=J_{\nu-1}(z)+J_{\nu+1}(z).\]
    Using this formula with $z=j_{\frac{m}{2}-1}$ and $\nu=\frac{m}{2}$ we obtain 
    \[c_1=m.\]
    Using again the formula with $\nu=\dfrac{m}{2}+1$ we obtain 
    \[c_2=m+2-\dfrac{j_{\frac{m}{2}-1}^2}{m}.\]
 Since $j_\nu^2<2(\nu+1)(\nu+3)$  (see \cite[section 5]{ismail1995bounds}), we obtain that 
 \begin{equation}\label{c_2 magg m/2}
     c_2>\dfrac{m}{2}.
 \end{equation}
Moreover (see \cite{L99s} and \cite{N14}) $\{c_k\}_{k\in\mathbb{N}}$ is decreasing, vanishing and the function $k\in\mathbb{R}\rightarrow c_k\in\RR$ can be extended to a convex function with domain $(-1,\infty)$.
    
\end{oss}

We are now ready to compute the second derivatives of the functionals $G$ and $F_q$ at $B_1$.
\begin{lemma} \label{der F eG} Let $\varphi\in C^{2,\gamma}(\partial B_1)$, with $\gamma>0$ and consider $G$ the functional defined in \eqref{function_G} and $F_q$ the functional defined in \eqref{fq}, with $q>0$. Using the notation introduced in \autoref{Teorema struttura}, it holds that 
\begin{equation}\label{der G}
\begin{multlined}
    \ell_2[G](\varphi,\varphi)=\\-\dfrac{1}{P(B_1))^{\frac{m}{m-1}}}\dfrac{\Lambda(B_1)}{m^2(m+2)(m-1)}\sum_{k=2}^{\infty}\sum_{l=1}^{N_k}\left[mk^2+(3m^2-4m)k-7m^2+7m+4c_k(m-1)\right]\varphi_{k,l}^2,
\end{multlined}
\end{equation}

and
\begin{equation}\label{der F_q}
\ell_2[F_q](\varphi,\varphi)=2\dfrac{T(B_1)^{q-1}\Lambda(0)}{m^2(m+2)\Vol(B_1)^{\alpha_q}}\sum_{k=2}^{\infty}\sum_{l=1}^{N_k}\left\{k\left[2-q(m+2)\right]+q(m+2)+2m-2-2c_k\right\}\varphi_{k,l}^2.
\end{equation}
\end{lemma}

\begin{proof}
    Since both $G$ and $F_q$ are invariant by translation and scaling, it is sufficient to compute the derivatives for $\varphi\in C^{2,\gamma}(\partial B_1)\cap T^0(\partial B_1)$. Hence, by \autoref{oss sferiche Ts} we assume that $\varphi_{k,l}=0$ for $k=0,1$.
    In the following, we will use the notation introduced in \autoref{notazione derivate seconde} and we consider the functions
\[g:t\in[0,1]\to G(B_{t\varphi})\in\RR,\qquad f_q:t\in[0,1]\to F_q(B_{t\varphi})\in\RR.\]

 By \autoref{oss calcolo l2}, we have that 
\[ \ell_2[G](\varphi,\varphi)=g''(0),\qquad\ell_2[F_q](\varphi,\varphi)=f_q''(0).\]

We start by computating $g''(0)$. Notice that $g(t)=\dfrac{\tau(t)\lambda(t)}{p(t)^{\frac{m}{m-1}}}$.

Differentiating two times and using that $\tau'(0)=p'(0)=\lambda'(0)=0$ we have that 
\[g''(0)=\dfrac{\lambda''(0)\tau(0)}{p(0)^{\frac{m}{m-1}}}+ \dfrac{\lambda(0)\tau''(0)}{p(0)^{\frac{m}{m-1}}}-\dfrac{m}{m-1}\dfrac{\lambda(0)\tau(0)}{p(0)^{\frac{m}{m-1}}}\left(\dfrac{p''(0)}{p(0)}\right).\]

Using \autoref{derivate seconde}, and the explicit values of the torsion of the unit ball 
\begin{equation}\label{torsione palla}
    T(B_1)=\frac{\omega_m}{m(m+2)},
\end{equation}
and of its perimeter, we obtain the following

\begin{equation*}
    \begin{split}
       g''(0)=\dfrac{1}{p(0)^{\frac{m}{m-1}}}\sum_{k=2}^\infty \sum_{l=1}^{N_k}\bigg\{\dfrac{\omega_m}{m(m+2)}\dfrac{\lambda(0)}{m\omega_m}\left[4k+2(m-1)-4c_k\right]-\dfrac{\lambda(0)}{m^2}\left[2k-(m+1)\right]\\
-\dfrac{m}{m-1}\lambda(0)\dfrac{\omega_m}{m(m+2)}\dfrac{1}{m\omega_m}\left[k^2+(m-2)k+(m-1)(m-2)\right]\bigg\}\varphi_{k,l}^2.
    \end{split}
\end{equation*}

By direct computation we have that 
\[g''(0)=-\sum_{k=2}^{\infty}\sum_{l=1}^{N_k}G_k\varphi_{k,l}^2.\]
where 
\[ G_k=\dfrac{1}{p(0)^{\frac{m}{m-1}}}\dfrac{\lambda(0)}{m^2(m+2)(m-1)}[mk^2+(3m^2-4m)k-7m^2+7m+4c_k(m-1)],\]
hence, we deduce \eqref{der G}.

We now want to compute $f_q''(0)$. We have that 
\[f_q(t)=\dfrac{\lambda(t)\tau(t)^q}{\vol^{\alpha_q}(t)}.\] Deriving twice and using that $\tau'(0)=\lambda'(0)=\vol'(0)=0$, we have that
\[f''_q(0)=\dfrac{\lambda''(0)\tau(0)^q}{\vol^{\alpha_q}(0)}+\dfrac{q\lambda(0)\tau(0)^{q-1}}{\vol^{\alpha_q}(0)}\tau''(0)-\alpha_q\dfrac{\lambda(0)\tau(0)^q}{\vol^{\alpha_q}\vol(0)}\dfrac{\vol''(0)}{\vol(0)}.\]

Using \autoref{derivate seconde} and equation \eqref{torsione palla}, we obtain 
\[f_q''(0)=\dfrac{\tau^{q-1}(0)}{\vol^{\alpha_q}(0)}\bigg\{\sum_{k=2}^\infty \sum_{l=1}^{N_k}\bigg\{\dfrac{\omega_m}{m(m+2)}\dfrac{\lambda(0)}{m\omega_m}\left[4k+2(m-1)-4c_k\right]\]
\[-q\dfrac{\lambda(0)}{m^2}\left[2k-(m+1)\right]-\alpha_q\lambda(0)\dfrac{\omega_m}{m(m+2)}\dfrac{(m-1)}{\omega_m}\bigg\}\varphi_{k,l}^2.\]
By direct computation, we have that 

\begin{equation}
    \begin{split}
        f''_q(0)=\dfrac{\tau(0)^{q-1}}{\vol^{\alpha_q}(0)}\dfrac{\lambda(0)}{m^2(m+2)}\sum_{k\geq 2}\sum_{l=1}^{N_k}\bigg\{2k[2-q(m+2)]+2(m-1)\\
        -4c_k+q(m+1)(m+2)-m\alpha_q(m-1)\bigg\}\varphi_{k,l}^2
    \end{split}
\end{equation}

Using that $m\alpha_q=q(m+2)-2$, we deduce 
\[f''_q(0)=\sum_{k=2}^\infty F^q_k\varphi_{k,l}^2.\]
where 
\[F_k^q= \dfrac{\tau(0)^{q-1}}{\vol^{\alpha_q}(0)}\dfrac{\lambda(0)}{m^2(m+2)}\{2k[2-q(m+2)]+2q(m+2)+4m-4-4c_k\},\]
hence, we have \eqref{der F_q}.

\end{proof}

\begin{oss}\label{oss cont l2}
    Notice that there exists a constant $r$ such that
    \[|\ell_2(G)(\varphi,\varphi)|\leq r\norm{\varphi}{H^1(\partial B_1)}^2.\]
Moreover, for all $q>0$ there exists a constant $r_q>0$ such that 
\[\text{if }q \not=\dfrac{2}{m+2}\quad|\ell_2(F_q)(\varphi,\varphi)|\leq r_q\norm{\varphi}{H^{\frac{1}{2}}(\partial B_1)}^2,\qquad \text{if }q=\dfrac{m}{m+2}\quad |\ell_2(F_q)(\varphi,\varphi)|\leq r_q\norm{\varphi}{L^{2}(\partial B_1)}^2.\]
Therefore, $\ell_2(G)$ admits a continuous extension to $H^1(\partial B_1)$, 
whereas $\ell_2(F_q)$ extends continuously to $H^{\frac{1}{2}}(\partial B_1)$ 
for $q \neq \frac{2}{m+2}$ and to $L^2(\partial B_1)$ when $q = \frac{2}{m+2}$.

\end{oss}

\begin{corollario}\label{Corollario G}
Let $G$ be the functional defined in \eqref{function_G}, then $B_1$ is a strictly stable shape in $H^1(\partial \om)$ for $-G$ in the sense of definition \autoref{defoptimal} and $-G$ satisfy the assumption \Csref at $B_1$ for $s=1$.
\end{corollario}
\begin{proof}
    We start by proving the result for $G$. By \autoref{der F eG} $\forall \varphi \in H^1(\partial B_1)$
    \[\ell_2(-G)(\varphi,\varphi)=\sum_{k=2}^{\infty}\sum_{l=1}^{N_k}G_k\varphi_{k,l}^2\]
where 
\[ G_k=\dfrac{1}{P(B_1)^{\frac{m}{m-1}}}\dfrac{\Lambda(B_1)}{m^2(m+2)(m-1)}[mk^2+(3m^2-4m)k-7m^2+7m+4c_k(m-1)].\]

We now aim to prove that for all $k\geq3 $
\begin{equation}\label{cresc G_k}
    G_k \geq G_2.
\end{equation}
Indeed, this inequality is equivalent to
\[
4(m-1)(c_k - c_2) \geq -\,m(k-2)(k+2+3m-4).
\]
Rewriting, this becomes
\begin{equation}\label{eq c_k 1}
    c_k \geq c_2 - \frac{1}{4} \frac{m}{m-1}({k + 3m - 2})(k-2).
\end{equation}

To estimate $c_k$ from below, we use the convex extension of the sequence 
$k \mapsto c_k$ described in \autoref{oss ck}. In particular, convexity implies
\[
c_k \ge c_2 + (c_2 - m)(k-2).
\]
Therefore, a sufficient condition for \eqref{eq c_k 1} is the inequality
\[
c_2 + (c_2 - m)(k-2)
    \;\ge\;
    c_2 - \frac{m}{4(m-1)}(k + 3m - 2)(k-2),
\]
which simplifies to
\[
m - c_2 \le \frac{m}{4(m-1)}(k + 3m - 2).
\]

This last inequality always holds for $k \ge 3$, since 
\( k + 3m - 2 > 0 \) and, by \eqref{c_2 magg m/2}, we have 
\( m - c_2 < \tfrac{m}{2} \). 
Hence, the desired inequality \eqref{cresc G_k} follows.

 We now observe that for all $k\geq 2$

\begin{equation}
    \begin{aligned}
        G_k \geq G_2 &= \dfrac{1}{P(B_1)^{\frac{m}{m-1}}}
        \dfrac{\Lambda(B_1)}{m^2(m+2)(m-1)}
        [4m+(3m^2-4m)2-7m^2+7m+4c_2(m-1)]\\
        &> \dfrac{1}{P(B_1)^{\frac{m}{m-1}}}
        \dfrac{\Lambda(B_1)}{m^2(m+2)(m-1)}
        [3m-m^2+4\dfrac{m}{2}(m-1)]\\
        &= \dfrac{1}{P(B_1)^{\frac{m}{m-1}}}
        \dfrac{\Lambda(B_1)}{m^2(m+2)(m-1)}
        [m^2+m] >0.
    \end{aligned}
\end{equation}

If $\varphi\in \mathcal{T}^1(\partial B_1)\setminus\{0\}$, there exist $k\geq 2$ and $l\in\{1,\dots N_k\}$ such that $\varphi_{k,l}\not=0$, hence we obtain 
\[\ell_2(-G)(\varphi,\varphi)>0.\]
Hence, by applying \autoref{oss cont l2}, we conclude that $B_1$ is a strictly stable shape for $-G$ in $H^1(\partial B_1)$. 
In addition, since $G_k$ has the same asymptotic behaviour as $(1+k^2)$ for $k$ that approaches $+\infty$, the functional $G$ satisfies assumption~\Csref{} with $s=1$.

\end{proof}

\begin{corollario}\label{Corollario F_q}
 Let $F_q$ be the functional defined in \eqref{fq}, with $q>0$, then the following properties hold.
    \begin{enumerate}[(i)]
        \item If $q<\dfrac{2}{m+2}$, $F_q$ verifies assumption \Csref on $B_1$ for $s=\frac{1}{2}$. Moreover and $B_1$ is a strictly stable shape in $H^{\frac{1}{2}}(\partial B_1)$ for $F_q$ (in the sense of \autoref{defoptimal}).
        \item If $q=\dfrac{2}{m+2}$, $F_q$ verifies assumption \Csref on $B_1$ for $s=0$. Moreover and $B_1$ is a strictly stable shape in $H^{0}(\partial B_1)$ for $F_q$ .
        \item If $\dfrac{2}{m+2}<q<q^*:=\dfrac{2}{m+2}\left(\dfrac{{j_{\frac{m}{2}}-1}^2}{m}-1\right)$, $B_1$ there exist $\varphi_1$ and $\varphi_2\in H^\frac{1}{2}(\partial B_1)$ such that 
        \[\ell_2(F_q)(\varphi_1,\varphi_1)>0\quad \text{and}\quad \ell_2(F_q)(\varphi_2,\varphi_2)<0.\]
        \item If $q=q^*:=\dfrac{2}{m+2}\left(\dfrac{j_{\frac{m}{2}-1}}{m}-1\right)$, for all  $\varphi\in H^\frac{1}{2}(\partial\om)$  \[\ell_2(F_q)(\varphi,\varphi)\leq0.\]
        Moreover, there exists $\varphi_1\in \mathcal{T}^{\frac{1}{2}}(\partial B_1)\setminus\{0\}$ such that  
        \[\ell_2(F_q)(\varphi_1,\varphi_1)=0.\]
        \item If $q>q^*$,  $-F_q$ verifies assumption \Csref on $B_1$ for $s=\frac{1}{2}$. Moreover and $B_1$ is a strictly stable shape in $H^{\frac{1}{2}}(\partial B_1)$ for $-F_q$.
    \end{enumerate}
\end{corollario}

\begin{proof}
By \autoref{der F eG}, we have 
\[\ell_2(F_q)(B_1)(\varphi,\varphi)=\sum_{k=2}^\infty\sum_{l=1}^{N_k}F^q_k \varphi_{k,l}^2.\]
where 
\[F_k^q= \alpha\{k[2-q(m+2)]+q(m+2)+2m-2-2c_k\},\quad \text{and}\quad\alpha=2\dfrac{T(B_1)^{q-1}}{\Vol^{\alpha_q}(B_1)}\dfrac{\Lambda(B_1)}{m^2(m+2)}.\]

First of all, we observe that if $q_1<q_2$, then for all $k\geq 2$
\begin{equation}\label{monotonia F_k in q}
    F_k^{q_1}>F_k^{q_2}.
\end{equation}

We analyse the different cases.
    \begin{enumerate}
        \item[(i)] Let $q\leq \dfrac{2}{m+2}$. Since $\set{c_k}$ is decreasing,  we have that if $k_1\leq k_2$, then
\[F^q_{k_1}\leq F_{k_2}^q.\]

Therefore, since $F_k^q$ is increasing in $k$ and decreasing in $q$, $F_k^q>0$ for all $k\geq 2$ and for all $q\leq \dfrac{2}{m+2}$ if and only if $F^\frac{2}{m+2}_2>0$. 
This inequality holds, since, by \autoref{oss ck},
\begin{equation}\label{termine noto positivo}
   F^\frac{2}{m+2}_2= \alpha(4+4m-4-4c_2)>0.
\end{equation}

By \eqref{termine noto positivo} and \autoref{oss cont l2}, if $q < \frac{2}{m+2}$, the ball $B_1$ is a strictly stable shape for 
$F_q$ in $H^{\frac{1}{2}}(\partial B_1)$.

Moreover, since $2-q(m+2)>0$, $F_k^q$  has the same asymptotic behaviour as $(1+k^2)^{\frac{1}{2}}$ for $k$ that approaches $+\infty$, the functional $F_q$ satisfies assumption~\Csref{} with $s=\frac{1}{2}$.

\item[(ii)]

Let $q=\dfrac{2}{2+m}$. Then $2-q(m+2)=0$ and for all $\varphi\in T^0(\partial B_1)$
\[\ell_2(F_q)(\varphi,\varphi)\geq r\norm{\varphi}{L^2(\partial B_1)}^2,\]
where 
\[r=\alpha\{2m-2c_2\}.\]
Therefore, using \autoref{oss cont l2} $F_q$ satisfies assumption~\Csref{} with $s=0$ and the ball $B_1$ is a strictly stable shape for 
$F_q$ in $H^{\frac{1}{2}}(\partial B_1)$.

\item [(iii)]
Let $\dfrac{2}{2+m}<q<q^*$ then we have that $[2-q(m+2)]<0$, so there exists $k^*$ such that for $k>k^*$,  $F_k^q<0$. 

Moreover, notice that $q*$ is such that 
\[F_2^{q*}=0.\]
Hence,by \eqref{monotonia F_k in q}, we deduce that for all $\dfrac{2}{m+2}<q<q*$ 
\[ F_2^q>F^{q*}_2=0.\]
Since $F_2^q>0$ and for $k$ sufficiently large $F_k^q<0$, $B_1$ is a saddle point for $F_q$.

\item[(iv)] 
Let $q\geq q^*$. We want to prove that for all $k>2$
\begin{equation}\label{mon Fkq}
    F_k^q<F_2^q
\end{equation}
This inequality holds if and only if 

\[c_k>\dfrac{k-2}{2}[2-q(m+2)]+c_2.\]
This inequality holds, indeed by \autoref{oss ck} 
\[\dfrac{k-2}{2}[2-q(m+2)]+c_2\leq (k-2)(c_2-m)+c_2=(k-1)(c_2-m)+c_1<c_k.\]
The last inequality follows by the existence of a convex extension of the function $k\in\mathbb{N}\to c_k\in\RR$. Hence, we have \eqref{mon Fkq}.\\

Furthermore, if $q=q^*$, $F_2^q=0$ and for all  $\varphi\in H^{\frac{1}{2}}(\partial \om)$
\[\ell_2(F_q)(\varphi,\varphi)\leq0.\]      Moreover if $\psi\in \mathcal{T}^{\frac{1}{2}}(\partial B_1)\setminus\{0\}$ such that for all $k\geq 3$ $\psi_{k,l}=0$, we have that   
\[\ell_2(F_q)(\psi,\psi)=0. \]

\item[(v)]
As a consequence, if $q>q^*$, for all  $\varphi\in \mathcal{T}^\frac{1}{2}(\partial \om)$ we have that  \[\ell_2(F_q)(\varphi,\varphi)>0.\]

Hence, using \autoref{oss cont l2}, the ball $B_1$ is a strictly stable shape for $-F_q$ in $H^{\frac{1}{2}}(\partial B_1)$.
Moreover, $-F_k^q$  has the same asymptotic behaviour as $(1+k^2)^{\frac{1}{2} }$ for $k$ that approaches $+\infty$. This ensures the functional $-F_q$ satisfies assumption~\Csref{} with $s=\frac{1}{2}$.

\end{enumerate}

\end{proof}

\subsection{Proof of  \autoref{Optimality of the ball} and \autoref{teor F_q}}
\begin{proof}[Proof of \autoref{Optimality of the ball}]
The thesis follows by applying \autoref{Teorema DL} to $-G$, as guaranteed by \autoref{Corollario G} and \autoref{proposizione IC}.

\end{proof}

\begin{proof}[Proof of \autoref{teor F_q}]
    The statement in (i) follows from the fact that, by 
    \autoref{Corollario F_q} and \autoref{proposizione IC}, the functional 
    $F_q$ satisfies on $B_1$ all the assumptions of \autoref{Teorema DL}.

    Statement (ii) is directly proved in \autoref{Corollario F_q}.

    Finally, the conclusion in (iii) again follows from the validity of the
    assumptions of \autoref{Teorema DL} for $-F_q$, ensured by 
    \autoref{Corollario F_q} and \autoref{proposizione IC}.
\end{proof}

\begin{oss}
    For the Kohler--Jobin case $q = \tfrac{2}{m+2}$ we only obtain
    $L^2(\partial B_1)$–coercivity together with improved continuity of the second
    derivative in $H^{1/2}(\partial B_1)$, which is not sufficient to deduce stability.

   If \( q = q^* \), coercivity fails, and therefore no conclusion on the local optimality of \( B_1 \) can be drawn.

\end{oss}

\section{Possible generalization}
\label{sec6}
Since the functional $F_q(\Omega)$ (defined in \eqref{fq}) can be seen as a generalization of the P\'olya functional (defined in \eqref{polyafunc}), it is natural to introduce the following scaling–invariant generalization of $G(\Omega)$:
\begin{equation}\label{def G_q}
    G_q(\om):=\dfrac{T(\om)^q\Lambda(\om)}{P(\om)^{\beta_q}},\text{ where } \beta_q=\dfrac{1}{m-1}[q(m+2)-2],
\end{equation}
for every $\Omega \in \mathcal{O}$ and all $q>0$.  

In particular, for $q=\tfrac{2}{m+2}$ the functional $G_q$ coincides with the Kohler–Jobin functional and with $F_q$, while for $q=1$ it reduces to $G$.

Our aim is to study the optimization of $G_q$ over the classes $\mathcal{O}$, $\mathcal{C}$, and $S_{\delta,\gamma}$, respectively defined in \eqref{openperimeter}, \eqref{open and convex} and \autoref{def Sdg}.

\subsection{Optimization in the class of open sets}
In this section we study the optimization of the functionals $G_q$ over the class $\mathcal{O}$.
\begin{prop}
Let $G_q$ be the functional defined in \eqref{def G_q}, and  $\mathcal{O}$ the class of domains defined in \eqref{openopt}. Then the following statements hold:
\begin{enumerate}
    \item If $q \leq \dfrac{2}{m+2}$, then 
    \[
        \min_{\mathcal{O}} G_q = G_q(B_1).
    \]
    \item If $q > \dfrac{2}{m+2}$, then 
    \[
        \inf_{\mathcal{O}} G_q = 0.
    \]
    \item If $q < 1$, then 
    \[
        \sup_{\mathcal{O}} G_q = +\infty.
    \]
    \item If $q \geq 1$, then 
    \[
        \sup_{\mathcal{O}} G_q < +\infty.
    \]
\end{enumerate}
\end{prop}

\begin{proof}
(1) Let $q \leq \dfrac{2}{m+2}$ and $\Omega \in \mathcal{O}$. Then
\[
    G_q(\Omega) = \big(T(\Omega)^{\frac{2}{m+2}}\Lambda(\Omega)\big)
    \left(\dfrac{1}{P(\Omega)^{-\frac{m+2}{m-1}}T(\Omega)}\right)^{\frac{2}{m+2}-q}.
\]
By the Kohler-Jobin inequality \eqref{K-J} and equation \eqref{sanvenper}, it follows that
\[
    G_q(B_1) \leq G_q(\Omega),
\]
hence $\min_{\mathcal{O}} G_q = G_q(B_1)$.

\medskip
(2) Let $q > \dfrac{2}{m+2}$. Consider $\Omega$ as the disjoint union of $N$ balls of fixed radius $\varepsilon > 0$ and $B_1$. Then, for $\varepsilon <1$,
\[
    G_q(\Omega) = \dfrac{\Lambda(B_1)(T(B_1)+N\varepsilon^{m+2}T(B_1))^q}{(P(B_1)+N\varepsilon^{m-1}P(B_1))} 
    = G_q(B_1)\dfrac{(1+N\varepsilon^{m+2})^q}{(1+N\varepsilon^{m-1})^{\beta_q}}.
\]
Choosing $N$ such that $\varepsilon^{-(m+2)} \leq N < 2\varepsilon^{-(m+2)}$, we obtain
\[
    G_q(\Omega) \leq G_q(B_1)\dfrac{3^q}{(1+\varepsilon^{-3})^{\beta_q}}.
\]
Since $\beta_q > 0$, this quantity tends to $0$ as $\varepsilon \to 0$, thus
\[
    \inf_{\mathcal{O}} G_q = 0.
\]

\medskip
(3) Let $q < 1$. Consider the sequence $\{\Omega_n\}$ of open sets introduced in the proof of \autoref{openopt} via a homogenization argument. Then, using the same estimates as in \autoref{openopt}, we have
\[
    \lim_{h \to \infty} G_q(\Omega_n) 
    = \lim_{h \to \infty} \dfrac{\Lambda(\Omega_n)T(\Omega_n)^q}{P(\Omega_n)^{\beta_q}}
    \geq \dfrac{\Lambda(\Omega,\mu)T(\Omega,\mu)^q}{P(\Omega)^{\beta_q}}
    \geq (c + \Lambda(\Omega))
    \left(\dfrac{\Vol(\Omega_\delta)^2}{\Vol(\Omega)}\dfrac{1}{\delta^{-2}+c}\right)^q
    \dfrac{1}{P(\Omega)^{\beta_q}}.
\]
Letting $\delta \to 0$ and $c \to \infty$, we conclude that $\sup_{\mathcal{O}} G_q = +\infty$.

\medskip
(4) Finally, if $q \geq 1$, then by \autoref{openopt} and equation \eqref{sanvenper}, for every $\Omega \in \mathcal{O}$,
\[
    G_q(\Omega) = G(\Omega)
    \left(P(\Omega)^{-\frac{m+2}{m-1}}T(\Omega)\right)^{q-1}
    < \dfrac{\Vol(B_1)}{P(B_1)^{\frac{m}{m-1}}}
    \left(P(B_1)^{-\frac{m+2}{m-1}}T(B_1)\right)^{q-1},
\]
which shows that $\sup_{\mathcal{O}} G_q < +\infty$.
\end{proof}

We recap the properties so far proven in \autoref{tab:Gq_propertiesopen}.

\begin{table}[H]
\centering
\renewcommand{\arraystretch}{1.3}
\begin{tabular}{|c|c|c|c|}
\hline
\( q \leq \frac{2}{m+2} \) & \( \frac{2}{m+2} < q < 1 \) & \( q = 1 \)& \( q > 1 \) \\
\hline
 \(\min_{\mathcal{O}} G_q = G_q(B_1)\) & \(\inf_{\mathcal{O}} G_q = 0\) &  \(\inf_{\mathcal{O}} G_q = 0\) & \(\inf_{\mathcal{O}} G_q = 0\) \\
 \hline
\(\sup_{\mathcal{O}} G_q = +\infty\) & \(\sup_{\mathcal{O}} G_q = +\infty\) & \(\sup_{\mathcal{O}} G_q  =\dfrac{\Vol(B_1)}{P(B_1)^{\frac{m}{m-1}}}\) & \(\sup_{\mathcal{O}} G_q  <+\infty\) \\
\hline
\end{tabular}
\caption{Summary of infimum and supremum values of the functionals \(G_q\) over the class $\mathcal{O}$ for different ranges of \(q\).}
\label{tab:Gq_propertiesopen}
\end{table}
\subsection{Optimization in the class of convex sets}

We now aim to optimize the functionals $G_q$ over the class of convex sets $\mathcal{C}$. To this end, we study the behavior of $G_q$ along sequences of thinning sets (defined in \autoref{thinset}).  

First, we observe that for every $\om \in \mathcal{C}$, the functional $G_q$ can be rewritten as the product of three terms:
\begin{equation}\label{G_q scrittura}
    G_q(\om)
    = \left(\frac{\Lambda(\om) T(\om)}{\Vol(\om)} \right)^q
      \left(\frac{\Lambda(\om)\, \Vol(\om)^2}{P(\om)^2}\right)^{1-q}
      \left(\frac{P(\om)^{\frac{m}{m-1}}}{\Vol(\om)}\right)^{2-3q}.
\end{equation}

It is known that for all $\om \in \mathcal{C}$ one has
\begin{equation}\label{bound Polya eig}
    \frac{\pi^2}{4m^2} < \frac{\Lambda(\om)\Vol(\om)^2}{P(\om)^2} < \frac{\pi^2}{4}.
\end{equation}
The lower bound was proved in \cite{M62} in the planar case, and it is attained by a sequence of thinning triangles. The result was generalized in \cite{DPDBG18} to higher dimensions. The upper bound was proven in \cite{polya1960two} in the planar case and is sharp for a sequence of thinning rectangles. Its generalization to higher dimensions was obtained in \cite{DPN14}, with quantitative refinements in \cite{AGS,AMPS}.  

Hence, using \eqref{bound polya}, \eqref{bound Polya eig}, and \eqref{G_q scrittura}, we obtain that for all $\om \in \mathcal{C}$:
\begin{equation}\label{bound gq convex}
\left(\frac{\pi^2}{4m(m+2)}\right)^q \left(\frac{\pi^2}{4m^2}\right)^{1-q} 
\left(\frac{P(\om)^{\frac{m}{m-1}}}{\Vol(\om)}\right)^{2-3q} 
< G_q(\om) 
< (1-\varepsilon)^q \left(\frac{\pi^2}{4}\right)^{1-q} 
\left(\frac{P(\om)^{\frac{m}{m-1}}}{\Vol(\om)}\right)^{2-3q}.
\end{equation}

Then, considering a sequence of thinning sets $\{\om_n\}$, and \eqref{vol  per thin}, we have that 
\[
\lim_n\frac{P(\om_n)^{\frac{m}{m-1}}}{\Vol(\om_n)}=0 .
\]
Hence, by \eqref{bound gq convex}, we have the following:
\begin{enumerate}[(i)]
    \item[(a)] If $q < \frac{2}{3}$, then
    \[
    \lim_{n \to \infty} G_q(\om_n) = +\infty \quad \text{and} \quad \sup_{\mathcal{C}} G_q = +\infty.
    \]
    \item[(b)] If $q > \frac{2}{3}$, then
    \[
    \lim_{n \to \infty} G_q(\om_n) = 0 \quad \text{and} \quad \inf_{\mathcal{C}} G_q = 0.
    \]
\end{enumerate}

Hence, we deduce the following proposition.

\begin{prop}
\label{6.2}
Let $G_q$ be the functional defined in \eqref{def G_q} and $\mathcal{C}$ the class defined in \eqref{open and convex}. The following properties hold.
\begin{enumerate}[(i)]
    \item If $q < \frac{2}{3}$, there exists $\om \in \mathcal{C}$ such that
    \[
    G_q(\om) = \min_{\mathcal{C}} G_q.
    \]
    \item If $q=\frac{2}{3}$, then 
    \[0<\inf_{\mathcal{C}}G_q<\sup_{\mathcal{C}}G_q<+\infty\]
    \item If $q > \frac{2}{3}$, there exists $\om \in \mathcal{C}$ such that
    \[
    G_q(\om) = \max_{\mathcal{C}} G_q.
    \]

\end{enumerate}
\end{prop}

\begin{proof}
If $q \ne \frac{2}{3}$, the behaviour of $G_q$ along sequences of thinning sets is straightforward:  
it tends to $0$ when $q > \tfrac{2}{3}$, and it diverges to $+\infty$ when $q < \tfrac{2}{3}$.  
Therefore, by arguments analogous to those used in \autoref{esistenza maX G},  
we immediately obtain statements \((i)\) and \((iii)\) of \autoref{6.2}.

The case $(ii)$, i.e. $q = \frac{2}{3}$, corresponds to the optimization of the functional
\[
\tilde{G}(\om) = \frac{T(\om)^2 \Lambda(\om)^3}{P(\om)^2} = G_{\frac{2}{3}}(\om)^3.
\]
From \eqref{G_q scrittura}, we have
\[
\tilde{G}(\om) = F(\om)^2 \left( \frac{\Lambda(\om) \Vol(\om)^2}{P(\om)^2} \right).
\]
Hence, by \eqref{bound polya} and \eqref{bound Polya eig}, for all $\om \in \mathcal{C}$ it holds that
\[
\left(\frac{\pi^2}{4m(m+2)}\right)^2 \left(\frac{\pi^2}{4m^2}\right)
< \tilde{G}(\om)
< (1-\varepsilon)^2 \left(\frac{\pi^2}{4}\right),
\]
hence, the thesis.
\end{proof}
\vspace{2mm}

We summarize the properties proven so far in ~\autoref{tab:Gq_propertiesconvex}.

\begin{table}[H]
\centering
\renewcommand{\arraystretch}{1.3}
\begin{tabular}{|c|c|c|c|}
\hline
\( q \leq \frac{2}{m+2} \) & \( \frac{2}{m+2} < q < \frac{2}{3} \) & \( q = \frac{2}{3} \)& \( q > \frac{2}{3} \) \\
\hline
 \(\min_{\mathcal{C}} G_q = G_q(B_1)\) & \(\min_{\mathcal{C}} G_q > 0\) &  \(\inf_{\mathcal{C}} G_q >0 \) & \(\inf_{\mathcal{C}} G_q = 0\) \\
 \hline
\(\sup_{\mathcal{C}} G_q = +\infty\) & \(\sup_{\mathcal{C}} G_q = +\infty\) & \(\sup_{\mathcal{C}} G_q  <+\infty\) & \(\max_{\mathcal{C}} G_q  <+\infty\) \\
\hline
\end{tabular}
\caption{Summary of infimum and supremum values of the functionals \(G_q\) over the class $\mathcal{C}$ for different ranges of \(q\).}
\label{tab:Gq_propertiesconvex}
\end{table}

\subsection{Optimization in the class of nearly spherical sets}
We have studied the local properties of the functional $G_q$ around the ball, obtaining the following result. 
\begin{teorema}\label{teor G_q}Let $G_q$ be the functional defined in \eqref{def G_q} for $q>0$, and let $\gamma>0$. Then the following properties hold:
\begin{enumerate}[(i)]
    \item If $q<\dfrac{2}{m+2}$,  then there exists $\delta>0$ and a constant $C=C(m,\delta)$  such that for every $B_h\in S_{\delta,\gamma}$ (defined in \ref{sdg})
\[G_q(B_h)-G_q(B_1)\geq C\norm{h}{H^1(\partial \om)}^2.\]
\item  If $\dfrac{2}{m+2}<q<q':=\dfrac{2}{m+2}\left[1+2\dfrac{m-1}{3m-1}\left(\dfrac{j_{\frac{m}{2}-1}}{m}-2\right)\right]$, $B_1$ is neither a local maximum or a local minimum, i.e. for all $\delta>0$ there exist $B_{h_1},\,B_{h_2}\in S_{\delta,\gamma} $ such that 
\[G_q(B_{h_1})<G_q(B_1)\quad \text{and}\quad G_q(B_{h_2})>G_q(B_1). \]
\item If $q>q' $,  then there exists $\delta>0$ and a constant positive $C=C(m,\delta)$  such that for every $B_h\in S_{\delta,\gamma}$
\[G_q(B_1)-G_q(B_h)> C\norm{h}{H^1(\partial \om)}^2.\]
\end{enumerate}
\end{teorema}

\begin{oss}
It is possible to observe for the functionals $G_q$
 the same local behaviour that we observed in \autoref{oss fq loc} for the functionals $F_q$. Indeed, in every dimension, it is possible to find an interval of values for $q$ in which $B_1$ is asaddle point. 
Moreover, we notice a distinct change in behaviour: while $B_1$ stops being both a local and a global minimum at \( q = \dfrac{2}{m+2} \), $B_1$ becomes a local maximum before turning into a global one as $q$ increases
\end{oss}

In order to prove \autoref{teor G_q}, we use the same techniques as in \autoref{optimization nearly spherical}, hence we prove the following lemmas.
\begin{lemma}
      Let $q>0$. $B_1$ is a critical shape for $G_q$ i.e. it holds that 
    \[\ell_1[G_q](B_1)\equiv0.\]
\end{lemma}

\begin{lemma}\label{lemma der seconde Gq}
Let $\varphi\in C^{2,\gamma}(\partial B_1)$, with $\gamma>0$ and consider $G_q$ the functional defined in \eqref{def G_q}, with $q>0$. Using the notation introduced in \autoref{Teorema struttura}, it holds that 
\begin{equation*}
    \begin{split}
        \ell_2[G_q](B_1)(\varphi,\varphi)&=-\dfrac{T^{q-1}(B_1)\Lambda(B_1)}{P^{\beta_q}(B_1)m^2(m+2)(m-1)}\sum_{k=2}^{\infty}\sum_{l=1}^{N_k}[(q(m+2)-2)k^2+\\
        &+(3m-4)(q(m+2-2))k-(m-1)(3q(m+2)+4m-6-4c_k)]\varphi_{k,l}^2.
    \end{split}
\end{equation*}
\
\end{lemma}
\begin{proof}
Since $G_q$ is invariant by translation and scaling, it is sufficient to compute the derivatives for $\varphi\in C^{2,\gamma}(\partial\om)\cap \mathcal{T}^0(\partial \om)$

We use the notation introduced in \autoref{notazione derivate seconde} and consider the function 
\[g_q:t\to G_q(\om_{th}).\] 
By \autoref{oss calcolo l2}, we have that $\ell_2(G_q)(\varphi,\varphi)=g_q''(0)$. Moreover, we notice that 

\[g_q(t)=\dfrac{\lambda(t)\tau(t)^q}{p(t)^{\beta_q}}.\] 
Hence, deriving twice and using that $\tau'(0)=\lambda'(0)=\vol'(0)=0$ we have that
\[g''_q(0)=\dfrac{\lambda''(0)\tau(0)^q}{p^{\beta_q}(0)}+\dfrac{q\lambda(0)\tau(0)^{q-1}}{p^{\alpha_q}(0)}\tau''(0)-\beta_q\dfrac{\lambda(0)\tau(0)^q}{\vol^{\alpha_q}(0)}\dfrac{p''(0)}{p(0)}.\]

Simplifying and using \autoref{derivate seconde} we obtain 
\[g_q''(0)=\dfrac{\tau^{q-1}(0)}{p^{\beta_q}(0)}\bigg\{\sum_{k=2}^\infty \sum_{l=1}^{N_k}\bigg\{\dfrac{\omega_m}{m(m+2)}\dfrac{\lambda(0)}{m\omega_m}\left[4k+2(m-1)-4c_k\right]\]
\[-q\dfrac{\lambda(0)}{m^2}\left[2k-(m+1)\right]-\beta_q\lambda(0)\dfrac{\omega_m}{m(m+2)}\dfrac{1}{m\omega_m}[k^2+(m-2)k+(m-1)(m-2)]\bigg\}\varphi_{k,l}^2.\]
Using that $\beta_q(m-1)=q(m+2)-1$, we obtain that
\[g_q''(0)=\dfrac{\tau^{q-1}(0)\lambda(0)}{p^{\beta_q}(0)m^2(m+2)(m-1)}\bigg\{\sum_{k=2}^\infty \sum_{l=1}^{N_k}\bigg\{(m-1)\left[4k+2(m-1)-4c_k\right]\]
\[-q(m-1)(m+2)\left[2k-(m+1)\right]-(q(m+2)-2)[k^2+(m-2)k+(m-1)(m-2)]\bigg\}\varphi_{k,l}^2.\]
By direct computation, we have that 
\[g_q''(0)=\sum_{k=2}^{\infty}\sum_{l=1}^{N_k}G_k^q\varphi_{k,l}^2.\]
where 
\[ G_k^q=\dfrac{\tau^{q-1}(0)\lambda(0)}{p^{\beta_q}(0)m^2(m+2)(m-1)}[(2-q(m+2))k^2+(3m-4)(2-q(m+2))k+(m-1)(3q(m+2)+4m-6-4c_k)].\]
\end{proof}

\begin{corollario}\label{Corollario G_q}
  Let $G_q$ be the functional defined in \eqref{def G_q}, with $q>0$, then the following properties hold.
 
    \begin{enumerate}[(i)]
        \item If $q<\dfrac{2}{m+2}$, $G_q$ verifies assumption \Csref on $B_1$ for $s=1$. Moreover and $B_1$ is a strictly stable shape in $H^1(\partial B_1)$ for $G_q$.
        \item If $q=\dfrac{2}{m+2}$, $G_q$ verifies assumption \Csref on $B_1$ for $s=0$. Moreover and $B_1$ is a strictly stable shape in $H^{0}(\partial B_1)$ for $G_q$.
        \item If $\dfrac{2}{m+2}<q<q':=\dfrac{2}{m+2}\left[1+2\dfrac{m-1}{3m-1}\left(\dfrac{j_{\frac{m}{2}-1}}{m}-2\right)\right],$ $B_1$ is a saddle point for $G_q$, i.e. there exist $\varphi_1$ and $\varphi_2\in H^1(\partial B_1)$ such that \[\ell_2(G_q)(\varphi_1,\varphi_1)>0\quad \text{and}\quad \ell_2(G_q)(\varphi_2,\varphi_2)<0.\]
        \item If $q=q'$, for all  $\varphi\in H^\frac{1}{2}(\partial B_1)$  \[\ell_2(G_q)(\varphi,\varphi)\leq0.\]
        Moreover, there exists $\varphi_1\in \mathcal{T}^{1}(\partial B_1)\setminus\{0\}$ such that  
        \[\ell_2(G_q)(\varphi_1,\varphi_1)=0.\]
        \item If $q>q'$, $-G_q$ verifies assumption \Csref on $B_1$ for $s=1$. Moreover and $B_1$ is a strictly stable shape in $H^{1}(\partial B_1)$ for $-G_q$.
    \end{enumerate}
\end{corollario}

\begin{proof}
By \autoref{lemma der seconde Gq}
\[\ell_2(G_q)(\varphi,\varphi)=\sum_{k=2}^\infty\sum_{l=1}^{N_k}G^q_k \varphi_{k,l}^2.\]
where 
\[ G_k^q=\alpha[(2-q(m+2))k^2+(3m-4)(2-q(m+2))k+(m-1)(3q(m+2)+4m-6-4c_k)],\] 
with
\[ \alpha=\dfrac{T^{q-1}(B_1)\Lambda(B_1)}{P^{\beta_q}(B_1)m^2(m+2)(m-1)}.\]
Notice that if $q_1<q_2$, then for all $k\geq 2$
\[G_k^{q_1}>G_k^{q_2}.\]

We analyse the different cases.
    \begin{enumerate}
        \item[(i)] Let $q\leq \frac{2}{m+2}$. Since $\set{c_k}$ is decreasing, we have that if $k_1\leq k_2$, then
\[G^q_{k_1}\leq G_{k_2}^q.\]

Therefore, since $G^q_k$ is increasing in $k$ and decreasing in $q$, $G^q_k>0$ for all $k\geq 2$ and $q\leq \dfrac{2}{m+2}$ if and only if $G^\frac{2}{m+2}_2>0$. 
This inequality holds, since, by \autoref{oss ck},
\begin{equation}\label{Gkpos}
   G_2^\frac{2}{m+2}= \alpha(6+4m-6-4c_2)>0.
\end{equation}

Since $\ell_2(G_q)$ admits a continuous extension to $H^1(\partial B_1)$, if $q < \frac{2}{m+2}$, the ball $B_1$ is a strictly stable shape for 
$G_q$ in $H^{1}(\partial B_1)$.

Moreover, since $2-q(m+2)>0$, $G_k^q$  has the same asymptotic behaviour as $(1+k^2)$ for $k$ that approaches $+\infty$, the functional $G_q$ satisfies assumption~\Csref{} with $s=1$.

\item[(ii)] If $q = \frac{2}{m+2}$, we have that $G_q=F_q$. Therefore, the thesis follows from \emph{(ii)} in \autoref{Corollario F_q}.

\item[(iii)] Let $\dfrac{2}{2+m}<q<q'$ then we have that $[2-q(m+2)]<0$, so there exists $k^*$ such that for $k>k^*$,  $G_k^q<0$. Moreover, notice that $q'$ is such that 
\[G_2^{q'}=0.\]
Hence, for all $\dfrac{2}{m+2}<q<q'$
\[ G_2^q>G^{q'}_2=0\]
Since $G_2^q>0$ and for $k$ sufficiently large $G_k^q<0$, $B_1$ is a saddle point for $G_q$.

\item[(iv)] 
Let $q\geq q'$. We want to prove that  for all $k>2$, 
\begin{equation}\label{mon Gqk}
G_k^q<G_2^q.
\end{equation}
This inequality holds if and only if 
\begin{equation}\label{intermedia decree}
4(m-1)(c_k-c_2)>(k-2)(2-q(m+2))(k+2+3m-4).
\end{equation}
For all $q\geq q'$, 
\[(2-q(m+2))\leq(2-q'(m+2))=-\dfrac{4(m-1)(m-c_2)}{3m-1},\]
hence \eqref{intermedia decree}, is true for all $q\geq q'$ if and only if 
\[c_k>c_2-(k-2)\left(\dfrac{k+3m-2}{3m-1}\right)(m-c_2).\]
This inequality holds, since by \autoref{oss ck} 
\[c_2-(k-2)\left(\dfrac{k+3m-2}{3m-1}\right)(m-c_2)<c_2-(k-2)(m-c_2)=c_2-(k-2)(c_1-c_2)<c_k.\]
The last inequality follows by the existence of a convex extension of the function $k\in\mathbb{N}\to c_k\in\RR$. Hence, we have \eqref{mon Gqk}.\\

Furthermore, if $q=q'$, $G_2^q=0$ and for all  $\varphi\in H^1(\partial \om)$
\[\ell_2(G_q)(\varphi,\varphi)\leq0.\]      Moreover if $\psi\in \mathcal{T}^{\frac{1}{2}}(\partial B_1)\setminus\{0\}$ such that for all $k\geq 3$ $\psi_{k,l}=0$, we have that   
\[\ell_2(G_q)(\psi,\psi)=0. \]

\item[(v)]As a consequence, if $q>q'$, for all  $\varphi\in \mathcal{T}^\frac{1}{2}(\partial \om)$ we have that  \[\ell_2(G_q)(\varphi,\varphi)>0.\]

Hence, since $\ell_2(G_q)$ admits a continuous extension to $H^1(\partial B_1)$, the ball $B_1$ is a strictly stable shape for $-G_q$ in $H^{\frac{1}{2}}(\partial B_1)$.
Moreover, $-G_k^q$  has the same asymptotic behaviour as $(1+k^2)$ for $k$ that approaches $+\infty$. This ensures the functional $-G_q$ satisfies assumption~\Csref{} with $s=1$. 

    \end{enumerate}
\end{proof}

The results of this section lead to the proof of \autoref{teor G_q} following the one of \autoref{teor F_q}.

\section*{Acknowledgments}
The authors were partially supported by Gruppo Nazionale per l’Analisi Matematica, la Probabilità e le loro Applicazioni
(GNAMPA) of Istituto Nazionale di Alta Matematica (INdAM).   \\
Vincenzo Amato was partially supported by  PNRR - Missione 4 “Istruzione e Ricerca” - Componente 2 “Dalla Ricerca all'Impresa” - Investimento 1.2 “Finanziamento di progetti
presentati da giovani ricercatori” CUP:F66E25000010006.

\bibliographystyle{plain}
\bibliography{biblio}
\Addresses
\end{document}